\documentclass[epic,eepic,11pt]{amsart}

\usepackage{amsfonts,mathrsfs}
\usepackage[mathscr]{eucal}

\usepackage{amssymb}
\usepackage{amscd}
\usepackage{fancyhdr}

\usepackage[small, nohug,heads=littlevee]{diagrams}
\diagramstyle[labelstyle=\scriptstyle]

\usepackage{euler,eucal}

\DeclareMathAlphabet{\mathbf}{T1}{ppl}{bx}{n}
\DeclareMathAlphabet{\mathrm}{T1}{ppl}{m}{n}

\numberwithin{equation}{section}

\newcommand\note[1]%
{$^\dagger$\marginpar{\footnotesize{$^\dagger${#1}}}}

\def\({\left(}
\def\){\right)}
\def\<{\left<}
\def\>{\right>}

\def\newi {\sqrt{-1}\, }

\newtheorem{theorem}{Theorem}[section]
\newtheorem{proposition}[theorem]{Proposition}
\newtheorem{lemma}[theorem]{Lemma}
\newtheorem{definition}[theorem]{Definition}

\newtheorem{corollary}[theorem]{Corollary}

\theoremstyle{definition}
\newtheorem{example}[theorem]{Example}
\newtheorem{remark}[theorem]{Remark}

\newcommand\lie{\mathfrak}
\renewcommand\t{\lie{t}}
\newcommand\g{\lie{g}}

\newcommand\bb[1]{{\text{\bf#1}}}

\newcommand\Z{\bb{Z}}

\newcommand\R{\mathbb{R}}
\newcommand\C{\mathbb{C}}

\newcommand\J{\mathcal{J}}
\newcommand\cO{\mathcal{O}}

\newcommand     {\comment}[1]   {}
\newcommand{\mute}[2] {}
\newcommand     {\printname}[1] {}

\newcommand\func[1]{\operatorname{\mathrm{#1}}}
\newcommand\funclim[1]{\operatorname*{\mathrm{#1}}}

\newcommand\type{\func{type}}

\renewcommand\exp{\func{exp}}

\renewcommand\lim{\funclim{lim}}

\newcommand\sur{\mathrel{\to\kern-1.8ex\to}}
\newcommand\iso{\mathrel{\hookrightarrow\kern-1.8ex\to}}

\newcommand\longhookrightarrow{\lhook\joinrel\longrightarrow}

\newcommand\longsur{\mathrel{\longrightarrow\kern-1.8ex\to}}
\newcommand\longiso{\mathrel{\longhookrightarrow\kern-1.8ex\to}}

\renewcommand\supset{\supseteq}

\begin{document}

\bibliographystyle{amsalpha}
\date{\today}
\title{The Equivariant cohomology theory of twisted generalized complex manifolds
}

\author{Yi Lin}
\maketitle
\begin{abstract} It has been shown recently by Kapustin and Tomasiello that
the mathematical notion of  Hamiltonian actions on twisted
generalized K\"ahler manifolds is in perfect agreement with the
physical notion of general $(2,2)$ gauged sigma models with
three-form fluxes. In this article, we study the twisted equivariant cohomology theory of
Hamiltonian actions on $H$-twisted generalized complex manifolds.
 If the manifold satisfies the $\overline{\partial}\partial$-lemma, we
establish the equivariant formality theorem. If in addition, the manifold satisfies the
generalized K\"ahler condition, we prove the Kirwan injectivity in
this setting. We then consider the Hamiltonian action of a torus on
an $H$-twisted generalized Calabi-Yau manifold and extend to this
case the Duistermaat-Heckman theorem for the push-forward measure.

 As a side result, we show in this paper that the generalized K\"ahler quotient
of a generalized K\"ahler vector space can never have a
(cohomologically) non-trivial twisting. This gives a negative answer
to a question asked by physicists whether one can construct $(2,2)$
gauged linear sigma models with non-trivial fluxes.
\end{abstract}

\section{Introduction}

Generalized complex geometry  was initiated by Hitchin \cite{H02}
and further developed by Gualtieri in \cite{Gua03}. It provides a
unifying framework for both symplectic geometry and complex
geometry, and turns out to be a useful geometric language
understanding various aspects of string theory. Generalized K\"ahler
geometry, the generalized complex version of K\"ahler geometry, was
introduced by Gualtieri, who also shows that it is equivalent to the
bi-Hermitian geometry which was first discovered by Gates, Hull and
Rocek \cite{GHR84} from the study of general $(2,2)$ supersymmetric
sigma models.

In \cite{LT05}, Tolman and the author proposed a definition of a
Hamiltonian action and a moment map in both generalized complex
geometry and generalized K\"ahler geometry. Recently, it has been shown by
Kapustin and Tomasiello \cite{KT06} that the mathematical notion of
Hamiltonian actions on generalized K\"ahler manifolds introduced in
\cite{LT05} is in perfect agreement with the expectations from
renormalization group flow, and corresponds exactly to the physical
notion of general $(2,2)$ gauged sigma models with three-form
fluxes.


In this paper, we study the equivariant cohomology theory of
Hamiltonian actions on twisted generalized complex and K\"ahler
manifolds. To better explain the ideas involved in this work, let us
first state an equivalent definition of Hamiltonian actions on
generalized complex manifolds. Suppose a compact Lie group $G$
with Lie algebra $\frak{g}$ acts on an $H$-twisted generalized
complex manifold $(M,\J)$, where $H$ is a closed three form on $M$. The action is said to be Hamiltonian if
there exists an equivariant smooth function $\mu: M\rightarrow
\frak{g}^*$, called the generalized moment map, a
$\frak{g}^*$-valued one form $\alpha \in \Omega^1(M,\frak{g}^*)$,
called the moment one form, such that $H+\alpha$ is equivariantly
closed (in the usual Cartan model) and such that the following
diagram commutes:

\[\begin{diagram}
 g &\rTo & C^{\infty}(M)\\
 &\rdTo &\dTo \\
 &  & C^{\infty}(TM \oplus T^*M),
\end{diagram}\]

where the horizontal map is given by $\xi\mapsto f^{\xi}$, the
vertical map is given by $ f \mapsto -\J df$, and the diagonal map
is given by $\xi \mapsto \xi_M+\alpha^{\xi}$. Therefore, morally
speaking, the Hamiltonian action of a compact Lie group $G$ not only
defines an action on the manifold $M$, but also defines an extended
action on the exact Courant algebroid $TM\oplus T^*M$. \footnote{
Indeed, the elements of extended group actions on exact Courant
algebroids has been developed in \cite{Hu05}, and, more generally,
in \cite{BCG05}.}

The equivariant cohomology theory has been an important topic in the
study of group actions on manifolds. However, the usual equivariant
cohomology theory encodes only the cohomological information of
group actions on manifolds, and does not give us any information on
the extended group actions on the exact Courant algebroid $TM\oplus
T^*M$. So when considering the consequence of the
$\bar{\partial}\partial$-lemma for Hamiltonian actions on
generalized complex manifolds, the author \cite{Lin06} was forced to
introduce two extensions of the usual equivariant de Rham cohomology
theory,  called the generalized equivariant cohomology and the
generalized equivariant Dolbeault cohomology, to encode the
information of the extended actions on $TM\oplus T^*M$.

In this paper, we explain that the generalized equivariant
cohomology introduced in \cite{Lin06} is actually the usual
equivariant cohomology twisted by the equivariantly closed three
form $H+\alpha$. We note that the equivariant cohomology theory of
general extended group actions on exact Courant algebroids has been
developed by Hu and Uribe \cite{HuUribe06}. \footnote{ We would also like to mention that the equivariant Lie algebroid
cohomology,  as introduced in \cite{BCRR05}, has been first studied in \cite[Sec. 5]{BCG05} in the
context of generalized complex geometry.} In particular, Hu and
Uribe proved that the twisted equivariant cohomology has the usual
expected properties of an equivariant cohomology theory such as
functoriality, Mayer-Vietoris lemma, Thom isomorphism, localization
theorem, etc.

In this context, the equivariant $\overline{\partial}\partial$-lemma
obtained in \cite{Lin06} can be best understood in the framework of
twisted equivariant cohomology theory. Let us consider the
Hamiltonian action of a compact Lie group on an $H$-twisted
generalized complex manifold satisfying the
$\overline{\partial}\partial$-lemma. As an immediate consequence of
the equivariant $\overline{\partial}\partial$-lemma, we obtain the
equivariant formality of the twisted equivariant cohomology
$H_G(M,H+\alpha)$. Now assume the compact Lie group is a torus $T$,
and assume that the manifold satisfies the generalized K\"ahler
condition, we establish the Kirwan injectivity theorem for the
twisted equivariant cohomology $H_T(M,H+\alpha)$.

It is an interesting question that if we can understand the
cohomology of the generalized K\"ahler quotients from the cohomology
of the Hamiltonian generalized K\"ahler manifold upstairs. We
believe that the equivariant formality theorem and Kirwan
injectivity theorem we established in this paper is useful for
answering such a question.  As a side result, we show that the
generalized K\"ahler quotient of a generalized K\"ahler vector space
can never have a (cohomologically) non-trivial twisting. This gives
a negative answer to a question asked by physicists whether one can
construct gauged linear sigma models with non-trivial fluxes.

As an application of the twisted equivariant cohomology,  we
investigate the Hamiltonian action of a torus $T$ on an $H$-twisted
generalized Calabi-Yau manifold $(M,\rho)$, where $\rho$ is an
$H$-twisted generalized Calabi-Yau structure, i.e., a nowhere
vanishing section of the canonical line bundle of the generalized
complex manifold $M$ such that $d\rho=H\wedge\rho$ for the closed
three form $H$.

Indeed, the Hamiltonian action on generalized Calabi-Yau manifolds
has been studied in the very recent works of Nitta, \cite{NY06},
\cite{NY07}. Under the assumption that both the twisting three form
$H$ and the moment one form $\alpha$ are zeroes, Nitta showed that
the generalized Calabi-Yau structure descends to the generalized
complex quotient. More interestingly, he observed that the Mukai
pairing $(\rho,\overline{\rho})$ gives rise to a volume form and so
a measure on the generalized Calabi-Yau manifold; and if the type of
the generalized Calabi-Yau structure is constant, then the push-froward
measure via the generalized moment map is absolutely continuous and
varies as a polynomial in the connected component of the regular
values of the generalized moment map. This has thus generalized the
usual Duistermaat-Heckman theorem to the setting of generalized
Calabi-Yau manifolds. \footnote{ However, it seems that Nitta's
proof \cite{NY07} has a small gap. Please see Remark
\ref{why-enlarge-coefficient-ring} for an explanation.}

However, the assumptions that both the twisting three form $H$ and
the moment one form $\alpha$ equal zero are not natural when one
studies group actions on generalized complex manifolds. Indeed, even
when the generalized complex manifold is untwisted, the moment one
form may come up very naturally \cite{LT05}. In this article, we
show that neither one of these two assumptions made in \cite{NY07}
are necessary, establishing the Duistermaat-Heckman theorem in the
setting of twisted generalized Calabi-Yau geometry in full
generality.  To achieve this we need to introduce the equivariant
differential forms with coefficients in the ring of formal power
series, which we discuss in the appendix.

Along the same line, we explain that the notion of Hamiltonian
action on generalized Calabi-Yau manifolds generalizes that of
Hamiltonian action on symplectic manifolds.  In addition, we give
concrete examples of Hamiltonian action on (non-symplectic)
generalized Calabi-Yau manifolds and calculate the density function
of the push-forward measure via the generalized moment map
explicitly.


Finally we would like to mention that there are four other groups
who have worked independently on the reduction and quotient construction
in generalized geometry: Bursztyn, Cavalcanti and Gualtieri \cite{BCG05}, S. Hu\cite{Hu05},
P. Xu and M. Sti\'{e}non \cite{SX05}, Vaisman, I. \cite{Va05}. In
particular, the connection between the work of Bursztyn, Cavalcanti and Gualtieri and
the gauged sigma model in physics is explained in \cite[Thm. 2.13]{BCG05}.

This paper is organized as follows. Section
\ref{generalities--twisted--equivariant--cohomology} gives a quick
review of some foundational aspects of twisted cohomology and
twisted equivariant cohomology. Section
\ref{generalized-complex-geometry} reviews elements of generalized
complex geometry. Section \ref{t-eq-theory-t-complex} establishes
the equivariant formality and Kirwan injectivity theorems for the
twisted equivariant cohomology of Hamiltonian actions on twisted
generalized K\"ahler manifolds. Section
\ref{generalized-complex-quotients} proves that the generalized
K\"ahler quotients of generalized vector spaces can never have a
non-trivial twisting.  Section \ref{DH-theorem-for-g-Calabi-Yau}
establishes the Duistermaat-Heckman theorem in the setting of
twisted generalized Calabi-Yau manifolds in full generality. Section
\ref{eg-hamiltonian-g-calabi-yau} presents concrete examples of
Hamiltonian action on generalized Calabi-Yau manifolds which are not
symplectic, and computes the density function of the push-forward
measure explicitly. Appendix A discusses the equivariant
differential forms with coefficients in the ring of formal power
series. \,\,\,\,\,\,\,\,\,\, \,\,\,\,\,\,\,\,\,\,
\,\,\,\,\,\,\,\,\,\,

{\bf Acknowledgement:} I would like to thank Tom Baird, Kentaro Hori, Lisa
Jeffrey, Anton kapustin, Yael Karshon, Eckhard Meinrenken, Reyer Sjamaar, and
Allessandro Tomasiello for useful conversations. And I would like to
thank Marco Gualtieri for useful comments on an earlier version of
this article.

A special thank goes to Lisa Jeffrey. The present work would not have
been possible without her kind support.

\section{Twisted equivariant
cohomology}\label{generalities--twisted--equivariant--cohomology}

\subsection{Twisted cohomology, non-equivariant case}

In this subsection, we briefly review the elements of twisted
cohomology theory. We refer to \cite{AS05} for an excellent account
of many fundamental aspects of the theory.

An $H$-twisted manifold $M$ is a smooth manifold in the usual sense
together with a closed three form $H$. The closed three form gives
rise to a twisted differential operator $d_H$ defined by
\[d_H\gamma=d\gamma-H\wedge \gamma, \,\,\,\gamma \in
\Omega(M).\] It is easy to check directly that $d_H^2=0$. We define
the $H$-{\bf twisted cohomology}
\[H(M, H)=\text{ker} d_H/\text{im}d_H.\]
Given a two form $\lambda$ on $M$, then we can form the twisted
differential $d_{H+d\lambda}$ and the group $H(M,H+d\lambda)$. There
is an isomorphism
\[H(M,H)\rightarrow H(M,H+d\lambda),\,\,\,[\alpha]\mapsto
[\text{exp}(\lambda)\alpha].\] So any two closed differential three
forms $H$ and $H'$, if representing the same cohomology class,
determine isomorphic groups $H(M,H)$ and $H(H,H')$. However, as
there are many two forms $\lambda$ such that $H'=H+d\lambda$, these
two groups are not uniquely isomorphic. Note also there is a
homomorphism of groups
\[H(M,H')\otimes H(M,H)\rightarrow H(M,H+H'), \,\,\, [\alpha]\otimes
[\beta]\mapsto [\alpha \wedge \beta].\] In particular, when $H'=0$,
this homomorphism gives $H(M,H)$ a natural $H(M)$-module structure:
if $[\alpha]\in H(M)$ and $[\beta]\in H(M,H)$, then $[\alpha \wedge
\beta]\in H(M,H)$.

\subsection{Equivariant de Rham theory}\label{Equivariant de Rham theory}
Let $G$ be a compact connected Lie group and let
$\Omega_G(M)=(S\g^*\otimes \Omega(M))^G$ be the Cartan complex of
the $G$-manifold $M$. For brevity we will write $\Omega=\Omega(M)$
and $\Omega_G=\Omega_G(M)$ when there is no confusion. By definition
an element of $\Omega_G$ is an equivariant polynomial from $\g$ to
$\Omega$ and is called an equivariant differential form on $M$. The
bi-grading of the Cartan complex is defined by
$\Omega^{ij}_G=(S^i\g^*\otimes \Omega^{j-i})^G$. It is equipped with
a vertical differential $1 \otimes d$, which is usually abbreviated
to $d$, and the horizontal differential $d'$, which is defined by
$d'\alpha (\xi)=-\iota_{\xi_M}\alpha(\xi)$. Here $\iota_{\xi_M}$
denotes inner product with the vector field $\xi_M$ on $M$ induced
by $\xi \in \g$. As a total complex, $\Omega_G$ has the grading
\begin{equation}\label{total-grading} \Omega_G^k=\bigoplus_{i+j=k}\Omega_G^{ij}\end{equation} and the total
differential $d_G=d+d'$. The total cohomology $\text{ker}d_G/
\text{im}d_G$ is the de Rham equivariant cohomology $H_G(M)$.



Now assume that the action of $G$ on $M$ is free. Choose a basis
$\xi_1,\xi_2,\cdots,\xi_k$ of $\frak{g}$. Then there exists
$\theta^1,\theta^2,\cdots, \theta^k \in \Omega^1(M)$, called the
connection elements, such that
$\iota_{\xi_{i,M}}\theta^j=\delta^i_j$, where $\delta^i_j$ is the
Kronecker symbol, and $\xi_{i,M}$ is the vector field on $M$ induced
by $\xi_i$. Any differential form $\gamma \in \Omega(M)$ can be
uniquely written as $\gamma=\gamma_{\text{hor}}+\sum_I \theta^I
\gamma_I$ such that $\iota_{\xi_{i,M}}\gamma_{\text{hor}}=0$ and
$\iota_{\xi_{i,M}}\gamma_I=0$, $1 \leq i\leq k$, where $I$ is a
multi-index, and $\gamma_{\text{hor}}$ is said to be the horizontal
component of the form $\gamma$.

 Let $c_{k,l}^i$ be the coefficient constants of the
Lie algebra $\frak{g}$, let
\begin{equation}\label{curvture-elements} c^i=d\theta^i+\dfrac{1}{2}c_{k,l}^i\theta^k\theta^l\end{equation} be the
curvature elements, and let $\Omega(M)_{\text{bas}}$ be the space of
basic forms in $\Omega(M)$, i.e., $G$-invariant forms which are
annihilated by the vectors fields generated by the group action. The
usual de Rham differential induces a differential $d$ on
 $\Omega_{\text{bas}}$ and the cohomology of the differential
 complex $(\Omega_{\text{bas}},d)$ is denoted by $H(\Omega_{\text{bas}})$.
Then the Cartan map $  \mathcal{C}:\Omega_G(M) \rightarrow
\Omega(M)_{\text{bas}}$, which is defined by
\begin{equation} \label{cartan-map} x^I\otimes \gamma \mapsto c^I \wedge
\gamma_{hor},\end{equation}
 induces an isomorphism from $H_G(M)$ to $H(\Omega_{\text{bas}})$.
 More specifically, there is a chain homotopy operator $Q:\Omega_G\rightarrow
 \Omega_G$ such that $d_GQ-Qd_G=\text{id}-\mathcal{C}$. Define $\partial_r:=\dfrac{\partial}{\partial
 x^r}$, $1 \leq r \leq k$, where $x^1, x^2, \cdots, x^k$ are
 inderterminates in the polynomial ring $S\frak{g}^*$, and define the operators $K$ and $R$ on $\Omega_G(M)$ by
\[K:= -\theta^r \partial_r,\,\,\,R:=d\theta^r \partial_r.\]
  Then the chain homotopy operator $Q$ is given by
  \begin{equation} \label{chain-homotopy} Q =KF(I+RF+(RF)^2+\cdots),\end{equation}
where $F$ is an operator on $\Omega_G(M)$ which preserves its
bi-grading. We refer to \cite[Chapter 5]{GS99} for more details.

 Let $\pi: M\rightarrow M/G$ be the quotient map. It is straightforward to check
 that the pull-back map
 $ \pi^*: H(M/G) \rightarrow H(\Omega_{\text{bas}})$ is an
 isomorphism. In any event, we have the following result.

\begin{theorem}\label{Cartan--Map}(\cite{GS99}) Assume the action of $G$ on $M$ is free. Then the Cartan map
\[ \Omega_G \rightarrow \Omega(M)_{\text{bas}},\,\, x^I\otimes \gamma \mapsto c^I \wedge \gamma_{hor},\]
induces a natural isomorphism from $H_G(M)$ to the ordinary
cohomology $H(M/G)$.

\end{theorem}

\subsection{Twisted equivariant cohomology}\label{Twisted equivariant
cohomology}

Throughout this subsection, we assume that $H_G$ is an equivariantly closed three form in the Cartan model $\Omega_G(M)$.
It gives rise to an $H_G$-twisted differential $d_{G,H_G}:= d_G -H_G\wedge $ on $\Omega_G(M)$ such that $d_{G,H_G}^2=0$.
We define \[ H(\Omega_G(M),H_G)=\text{ker} d_{G,H_G}/\text{im}d_{G,H_G}.\]

Let $\widehat{\Omega}_G=\left(\widehat{S\g^*}\times  \Omega(M)\right)^G$. Here
$\widehat{S\g^*}=\C[[x_1,x_2,\cdots,x_n]]$ is the $a$-adic completion of the polynomial
ring $S\g^*$ with respect to the ideal $a=(x_1,x_2,\cdots,x_n)$.  It is easy to see that the $H_G$-twisted equivariant differential on $\Omega_G$
extends naturally to a differential on $\widehat{\Omega}_G$ which squares to zero. For brevity, this differential on $\widehat{\Omega}_G$ is also denoted by $d_{G,H_G}$. The $H_G$-{\bf
twisted equivariant cohomology} is defined to be
\[ H_G(M, H_G)=\text{ker}(\widehat{\Omega}_G \xrightarrow{ d_{G,H_G}}\widehat{\Omega}_G)\diagup\text{im}(\widehat{\Omega}_G \xrightarrow{ d_{G,H_G}}\widehat{\Omega}_G).\]

In the special case $H_G=0$, the $H_G$-twisted equivariant cohomology is
just $H(\widehat{\Omega}_G(M),d_G)$, the equivariant cohomology with coefficients in the formal power series ring as we discussed in Appendix \ref{generalized-eq-diff-forms}.

Similar to the non-equivariant case, we have the following results.

\begin{lemma}\label{basic-facts-twisted-equivariant-cohomology}
If $\lambda$ is an equivariant differential two form, then there is
an isomorphism
\[H_G(M,H_G)\rightarrow H_G(M,H_G+d_G\lambda),\,\,\, [\alpha]\mapsto
[\text{exp}(\lambda)\alpha].\] Thus two equivariantly closed three
forms $H_G$ and $H_G'$ representing the same equivariant cohomology
class determine the isomorphic groups $H_G(M,H_G)$ and
$H_G(M,H_G')$.

\end{lemma}

\begin{lemma}\label{module-structure} Let $H_G$ and $H_G'$ be two
equivariantly closed three forms. If $\gamma$ is $d_{G,H_G}$-closed,
and $\gamma'$ is $d_{G,H_G'}$-closed, then $\gamma' \wedge\gamma$ is
$d_{G,H_G+H_G'}$-closed. Consequently there is a homomorphism
\[H_G(M,H'_G)\otimes H_G(M,H_G)\rightarrow H_G(M,H_G+H'_G),
\,\,\,[\gamma']\otimes [\gamma]\mapsto [\gamma' \wedge \gamma].\] In
particular, when $H_G'=0$, the homomorphism
\[H(\widehat{\Omega}(M)_G, d_G)\otimes H_G(M,H_G)\rightarrow H_G(M,H_G),
\,\,\,[\gamma']\otimes [\gamma]\mapsto [\gamma' \wedge \gamma]\]
defines an $H(\Omega_G(M),d_G)$-module structure on $H_G(M,H_G)$. \end{lemma}

Note that since $H(\widehat{\Omega}(M)_G, d_G)$ is
a $(\widehat{S\frak{g}^*})^G$-module, c.f. Appendix \ref{generalized-eq-diff-forms}, $H_G(M,H_G)$ carries a
$(\widehat{S\frak{g}^*} )^G$-module structure. Moreover, since $(S\g^*)^G$ is a subring of $(\widehat{S\frak{g}^*} )^G$,
$H_G(M,H_G)$ is also a $(S\frak{g}^* )^G$-module.

Now we describe the relationship between the twisted equivariant cohomology $H_G(M,H_G)$ and $H(\Omega_G(M),H_G)$.
The sequence of modules
\[ S\frak{g}^*\otimes \Omega(M) \supset a \left(S\frak{g}^*\otimes \Omega(M)\right)
\supset \cdots \supset a^k \left(S\frak{g}^*\otimes
\Omega(M)\right)\supset \cdots \] defines an $a$-adic topology on
$S\frak{g}^*\otimes \Omega(M)$, and so defines an $a$-adic topology on its subspaces
$\Omega_G=\left(S\frak{g}^*\otimes \Omega(M)\right)^G$ and $\text{ker}(\Omega_G \xrightarrow{ d_{G,H_G}}\Omega_G) $.
Then from the quotient map
\[ \text{ker}(\Omega_G \xrightarrow{ d_{G,H_G}}\Omega_G) \rightarrow H(\Omega_G(M),H_G) \]
we get an $a$-adic topology on $H(\Omega_G(M),H_G)$. It is straightforward to check by definition that
 the completion of $H(\Omega(M),H_G)$ with respect to this topology also carries a $(\widehat{S\g^*})^G$-module structure.
  The following result is a simple consequence of \cite[Prop. 10.2]{AM99}.

\begin{proposition} \label{completion-of-cohomology} There is a natural $(\widehat{S\g^*})^G$-module isomorphism from the
$H_G$-twisted equivariant cohomology $H_G(M,H_G)$ to the completion of
$H(\Omega_G(M),H_G)$ with respect to the $a$-adic topology we described above.
\end{proposition}



We refer to \cite{HuUribe06} for a detailed treatment of the
foundational aspects of the twisted equivariant cohomology. For our
purpose, we will need the following result later on.

\begin{theorem}(\textbf{Localization Theorem})(\cite{HuUribe06})\label{t--localizationI} Suppose a compact
connected torus $T$ acts on a compact manifold $M$. Then the kernel
and cokernel of the canonical map
\[ i^*: H_{T}^*(M,H_T) \rightarrow H_{T}^*(M^T,i^*H_T),\]
have support contained in $\bigcup_H\frak{h}$,  where $M^T$ is the
fixed points set of the torus $T$ action, $H$ runs over the
stabilizers $\neq T$, and $\frak{h}$ is the Lie algebra of $H$.
\end{theorem}

\begin{remark}Let $\t^*$ be the dual space of the Lie algebra $t$ of $T$. Then $S\t^*=\C[x_1,x_2,\cdots,x_n]$ is a polynomial ring in $n$ variables.
The support of an $S\t^*$-module $A$ is defined to be $\text{supp}(A)=\bigcap_{\{f \in S\t^*\mid f\cdot A=0\}}V_f$, where
$V_f=\{x\in t\mid f(x)=0\}$. The support of $H_{T}^*(M,H_T)$ discussed in Theorem \ref{t--localizationI} is the support of
it as an $S\t^*$-module.
\end{remark}

                   \section{ GENERALIZED COMPLEX GEOMETRY}\label{generalized-complex-geometry}
  Let $V$ be an $n$ dimensional vector space. There is a natural bi-linear pairing
of type $(n, n)$ on $V\oplus V^*$ which is defined by
                            \[    \langle X +\alpha , Y +  \beta \rangle = \dfrac{1}{2}(\beta (X) +
                            \alpha(Y)).\]
A \textbf{generalized complex structure} on a vector space $V$ is an
orthogonal linear map $\mathcal{J}: V\oplus V^* \rightarrow V\oplus
V^*$ such that $\mathcal{J}^2=-1$. Let $L \subset V_{\C}\oplus
V^*_{\C}$ be the $\sqrt{-1}$ eigenspace of the generalized complex
structure $\mathcal{J}$. Then $L$ is maximal isotropic and $L \cap
\overline{L} = \{0\}$. Conversely, given a maximal isotropic
$L\subset V_{\C}\oplus V^*_{\C}$  so that  $L \cap \overline{L} =
\{0\}$, there exists an unique generalized complex structure
$\mathcal{J}$ whose $\sqrt{-1}$ eigenspace is exactly $L$.

 Let $\pi \colon V_\C \oplus V^*_\C  \to V_\C$ be the
natural projection. The {\bf type} of $\mathcal{J}$ is the
codimension of $\pi(L)$ in $V_\C$, where $L$ is the $\sqrt{-1}$
eigenspace of $\mathcal{J}$.


Let $\sigma$ be the linear map on $\wedge V^*$ which acts on
decomposables by
\begin{equation}\label{anti-auto1} \sigma (v_1\wedge v_2 \wedge
\cdots \wedge v_q)=v_q \wedge v_{q-1}\wedge \cdots \wedge
v_1,\end{equation} then we have the following bilinear pairing,
called the Mukai pairing \cite{Gua07}, defined on $\wedge V^*$:
\[ (\xi_1,\xi_2)=(\sigma(\xi_1)\wedge \xi_2)_{\text{top}}, \] where
$(\cdot,\cdot)_{\text{top}}$ indicates taking the top degree
component of the form.

Let $B \in \wedge^2 V^*$ be a two-form. For any $\xi \in \wedge
V^*$, throughout this paper we will denote by $e^B \xi$ the wedge
product $e^B\wedge \xi$. The following result was shown in
\cite[Prop. 1.12]{Gua07}.

\begin{proposition}\label{invariance-mukai-pairing}
\[ (e^B \xi_1, e^B\xi_2)=(\xi_1,\xi_2),\,\,\,\text{ for any
$\xi_1,\xi_2 \in \wedge V^*$},\] where $(\cdot,\cdot)$ denotes the
Mukai pairing.
\end{proposition}


The Clifford algebra of $V_{\C}\oplus V_{\C}^*$ acts on the space of
forms $\wedge V_{\C}^*$ via
\[(X+\xi)\cdot \varphi=\iota_X\varphi+\xi\wedge\varphi.\]

Since $\mathcal{J}$ is skew symmetric with respect to the natural
pairing on $V \oplus V^* $,  $\mathcal{J}\in so(V \oplus V^*) \cong
\wedge^2(V \oplus V^*)\subset CL (V \oplus V^*)$. Therefore there is
a Clifford action of $\mathcal{J}$  on the space of  forms $\wedge
V^*$  ( and so on its complexification $\wedge V^*_{\C}$) which
determines an alternative grading : $ \wedge
V_{\C}^*=\bigoplus_{\begin{subarray} k
\end{subarray}}U^k,$ where $U^k$ is the $-k\sqrt{-1}$ eigenspace of
the Clifford action of $\J$.

Note also that a classical result of Chevalley \cite{Che97} asserts
that
\[\{\varphi \in \wedge V_{\C}^*\mid (X+\xi)\cdot \varphi=0, X+\xi \in L\}\]
is a one dimensional vector space in the space of exterior forms
$\wedge V_{\C}^*$. Any form in this one dimensional subspace of
$\wedge V_{\C}^*$ is said to be a pure spinor form associated to the
generalized complex structure $\mathcal{J}$. On the other hand,
given a form $\varphi \neq 0 \in \wedge V_{\C}^*$, $L_{\varphi}:=\{
X+\xi \in V_{\C}\oplus V_{\C}^*\mid (X+\xi)\cdot \varphi=0\}$ is
always an isotropic space. If in addition,  $L_{\varphi}$ is maximal
isotropic, and if $(\varphi, \overline{\varphi})\neq 0$, then
$L_{\varphi}\cap \overline{L}_{\varphi}=\{0\}$ and there exists an
unique generalized complex structure $\mathcal{J}$ whose $\sqrt{-1}$
eigenspace is exactly $L_{\varphi}$.


Let $M$ be a manifold of dimension $n$. There is a natural pairing
of type $(n,n)$ which is defined on $TM\oplus T^*M$ by
\[ \langle X+\alpha, Y+\beta \rangle
=\dfrac{1}{2}\left(\beta(Y)+\alpha(X)\right),\] and which extends
naturally to $T_{\C}M\oplus T_{\C}^*M$.

For a closed three form $H$, the $H$-\textbf{twisted Courant
bracket} of $T_{\C}M\oplus T^*_{\C}M$ is defined by the identity
\[
[X+\xi,Y+\eta]=[X,Y]+L_X\eta-L_Y\xi-\dfrac{1}{2}d\left(\eta(X)-\xi(Y)\right)+\iota_Y\iota_XH.\]

A \textbf{generalized almost complex structure} on a manifold $M$ is
an orthogonal bundle map $\mathcal{J}:TM\oplus T^*M \rightarrow
TM\oplus T^*M$ such that $\mathcal{J}^2=-1$. Moreover, $\mathcal{J}$
is an $H$-\textbf{twisted generalized complex structure} if the
sections of the $\sqrt{-1}$ eigenbundle of $\mathcal{J}$ is closed
under the $H$-twisted Courant bracket. The {\bf type} of
$\mathcal{J}$ at $m \in M$ is the type of the restricted generalized
complex structure on $T_m M$.

The Clifford algebra of $C^{\infty}(TM\oplus T^*M)$ with the natural
pairing acts on differential forms by
\[ (X+\xi)\cdot \varphi=\iota_X\varphi+\xi\wedge \varphi.\]
If $L\subset T_{\C}M\oplus T^*_{\C}M$ is the $\sqrt{-1}$ eigenbundle
of a generalized complex structure $\J$, then the differential forms
annihilated by the Clifford action of the sections of $L$ span a line
bundle sitting inside $\wedge T^*_{\C}M$, which is called the
\textbf{canonical line bundle} of the generalized complex manifold
$(M,\J)$.

Since $\mathcal{J}$ can be identified with a smooth section of the
Clifford bundle $CL(TM\oplus T^*M)$, there is a  Clifford action of
$\mathcal{J}$ on the space of differential forms. Let $U^k$ be the
$-k\sqrt{-1}$ eigenbundle of $\mathcal{J}$. \cite{Gua07} shows that
there is a grading of the differential forms:
   \begin{equation}\label{alternative-grading}
       \Omega^*(M) =\Gamma( U^{-n})\oplus\cdots\oplus \Gamma(U^0) \oplus \cdots
  \oplus \Gamma(U^n) .\end{equation}

It has been shown (See e.g. \cite{Gua07} and \cite{KL04}) that the
integrability of an $H$-twisted generalized complex structure
$\mathcal{J}$ implies that
\[ d_H=d-H\wedge : \Gamma(U^k) \rightarrow \Gamma(U^{k-1})\oplus \Gamma(U^{k+1}),\]
 which gives rise to operators $\partial$ and $\bar{\partial}$
 via the projections
\[\partial: \Gamma(U^k)\rightarrow \Gamma(U^{k-1}),
 \,\,\,\bar{\partial}:\Gamma(U^k)\rightarrow \Gamma(U^{k+1}).\]


In this context, a twisted generalized complex manifold is said to
satisfy the $\bar{\partial}\partial$-lemma if and only if
\[ \text{ker}\partial \cap \text{im} \bar{\partial} =\text{im}\partial
\cap \text{ker} \bar{\partial} =\text{im}\bar{\partial}\partial.\]
Let $\Omega_{\overline{\partial}}(M)=\Omega(M)\cap \text{ker}
\overline{\partial}$. Since $d_H$ anti-commutes with
$\overline{\partial}$, $(\Omega_{\overline{\partial}}, d_H)$ is a
differential complex with the differential $d_H$. The following
result is a simple consequence of the
$\bar{\partial}\partial$-lemma.

\begin{proposition}(c.f., \cite[Thm. 4.1]{Ca06})\label{consequence-ddbar-lemma} If the generalized complex manifold
$(M, \mathcal{J})$ satisfies the $\bar{\partial}\partial$-lemma,
then the inclusion of the complex of $\overline{\partial}$-closed
differential forms $\Omega_{\overline{\partial}}(M)$ into the
complex of differential forms $\Omega(M)$ induces an isomorphism in
cohomology:
\[i: (\Omega_{\overline{\partial}}(M), d_H)\hookrightarrow (\Omega(M), d_H), \,\,\,i^*: H(\Omega_{\overline{\partial}}(M)) \cong H(M,H).\]

\end{proposition}


 An $H$-\textbf{twisted generalized Calabi-Yau} structure $\J$ is a generalized
complex structure such that its canonical line bundle admits a
nowhere vanishing $d_H$-closed section $\rho$ . On the other hand,
given a nowhere vanishing $d_H$-closed differential form $\rho$ such
that $(\rho,\overline{\rho})\neq 0$, there exists a generalized
Calabi-Yau structure $\J$ such that $\rho$ is the pure spinor
associated to $\J$. By the abuse of the notation, we will also call
such a differential form $\rho$ a generalized Calabi-Yau structure.

Let $B$ be a closed two-form on a manifold M, and consider the
orthogonal bundle map defined by

\[ e^B=\left(\begin{matrix} 1 & 0\\ B&1 \end{matrix}\right): TM\oplus T^*M\rightarrow TM \oplus T^*M,\]
where $B$ is regarded as a skew-symmetric map from $TM$ to $T^*M$.
This map preserves the canonical pairing $<\cdot,\cdot>$ and the
$H$-twisted Courant bracket. As a consequence, if $\mathcal{J}$ is
an $H$-twisted generalized complex structure on M, then
$\mathcal{J}_B := e^B\mathcal{J} e^{-B}$ is another $H$-twisted
generalized complex structure on M, called the B-transform of
$\mathcal{J}$.


The effect of $B$-transforms on the canonical line bundle of a
generalized complex structure has been studied in \cite{Ca06}.

\begin{lemma}(\cite{Ca06} )\label{Btransformoperator} Let
$B$ be a closed two form and let $\mathcal{J}_B$ be the
$B$-transform of the generalized Calabi-Yau structure $\mathcal{J}$
with a nowhere vanishing closed pure spinor $\rho$. Then $\J_B$ is a
generalized Calabi-Yau structure with a nowhere vanishing closed
pure spinor $e^{-B}\rho$. \end{lemma}





\begin{example} (\cite{H02}, \cite{Gua07})\label{symplectic-g-calabi-Yau}
\begin{itemize}
\item Let $(V, \omega)$ be a $2n$ dimensional symplectic vector
space. Then the map $\mathcal{J}:V\oplus V^*\rightarrow V \oplus
V^*$ defined by
\begin{equation} \label{symplectic}
\mathcal{J}=\left(\begin{matrix} 0 & -\omega^{-1} \\\omega
&0\\\end{matrix} \right)\end{equation} is a generalized complex
structure on $V$.

\item Now let $(M,\omega)$ be a  $2n$ dimensional symplectic manifold.
 Then (\ref{symplectic}) defines
a generalized complex structure $\J_{\omega}$ with the $\sqrt{-1}$
eigenbundle $L=\{X-\sqrt{-1}\iota_X\omega \,\vert\, X \in
T_{\C}M\}$. It is easy to check that the canonical line bundle
admits a nowhere vanishing closed section $e^{i\omega}$. And so
$\J_{\omega}$ is a generalized Calabi-Yau structure. On the other
hand, the following lemma shows that any nowhere vanishing closed
section of the canonical line bundle of $\J_{\omega}$ must be of the
form $ae^{i\omega}$ for some non-zero constant $a$.
\end{itemize} \end{example}

\begin{lemma}\label{symplectic-g-calabi-yau} Let $(M,\omega)$ be a $2n$ dimensional connected symplectic manifold
and $\J_{\omega}$ the generalized complex structure defined as in
(\ref{symplectic}). Suppose $\rho$ is a nowhere vanishing closed
pure spinor associated to the generalized Calabi-Yau structure
$\J_{\omega}$. Then we have $\rho=a e^{i\omega}$ for some non-zero
constant $a$. \end{lemma}
\begin{proof} By assumption we have $\rho= f e^{i \omega}$ for some nowhere vanishing function
$f \in C^{\infty}(M)$. Since $d\rho=0$, we get $df \wedge
e^{i\omega}=0$. In particular, we have
\[ (i\omega)^{n-1} \wedge df =0.\]
It is well known that the Lefschetz map \[
L_{\omega}^{n-1}:\Omega^1(M) \rightarrow \Omega^{2n-1}(M),\,\,\,
\eta \mapsto \omega^{n-1}\wedge \eta\] is an isomorphism, c.f.,
\cite[Corollary 2.7]{yan;hodge-structure-symplectic}. It follows
that $df=0$ and so $f=a$ for some non-zero constant $a$. This
finishes the proof. \end{proof}

It is important for our purpose to notice that given a $2n$
dimensional $H$-twisted generalized Calabi-Yau manifold $M$, the
generalized Calabi-Yau structure $\rho$ on $M$ gives rise to a
volume form
\begin{equation}\label{volume-form} \frac{(-1)^n}{(2i)^n}(\rho,\overline{\rho}),\end{equation} where
$(\cdot,\cdot)$ denotes the Mukai pairing. Note that if $M$ is a
symplectic manifold with a symplectic form $\omega$ and the
generalized Calabi-Yau structure is given by $e^{i\omega}$, then up
to a normalization factor the volume form we defined above coincides
with the symplectic volume $\frac{1}{n!}\omega^n$.

A manifold $M$ is said to be an $H$-twisted generalized K\"ahler
manifold if it has two commuting $H$-twisted generalized complex
structures $\J_1$, $\J_2$ such that $\langle -\J_1\J_2\xi,\xi\rangle
>0$ for any $\xi \neq 0 \in C^{\infty}(T_{\C}M \oplus  T^*_{\C}M)$,
where $\langle\cdot,\cdot \rangle$ is the canonical pairing on
$T_{\C}M \oplus  T^*_{\C}M$. The following remarkable result is due
to Gualtieri.

\begin{theorem}\label{Gua'slemma}  (\cite{Gua04})  Assume that $(M,\J_1,\J_2)$ is a
  compact $H$-twisted generalized K\"ahler manifold. Then it
   satisfies the $\bar{\partial}\partial$-lemma
with respect to both $\J_1$ and $\J_2$.
\end{theorem}

\section{The twisted equivariant cohomology of Hamiltonian  actions on twisted generalized complex manifolds }
\label{t-eq-theory-t-complex}

First we recall the definition of Hamiltonian actions on
$H$-twisted generalized complex manifolds given in \cite{LT05}.

\begin{definition}\footnote{Indeed, Condition (b) was not imposed
in \cite[Definition A.2.]{LT05}. However, in order to make the quotient construction work,
Tolman and the author made it clear in \cite[Prop. A.7, A.10]{LT05} that $H+\alpha$ must be equivariantly closed in the usual
Cartan model.}
\cite{LT05}) \label{deftmm}
Let a compact Lie group $G$ with Lie algebra $\g$ act on a manifold
$M$, preserving an $H$-twisted generalized complex structure
$\mathcal{J}$, where $H \in \Omega^3(M)^G$ is closed. Let $L \subset
T_\C M \oplus T^*_\C M$ denote the $\sqrt{-1}$ eigenbundle of $\J$.
The action of $G$ is said to be Hamiltonian if there exists a smooth
equivariant function $\mu:M \rightarrow \frak{g}^*$,  called the
{\bf generalized moment map}, and a one form $\alpha \in
\Omega^1(M,\g^*)$, called the {\bf moment one form},  so that
\begin{itemize}
\item [a)]
$\xi_M - \newi (d\mu^\xi+  \newi \alpha^\xi)$ lies in
$C^{\infty}(L)$ for all $\xi \in \g$, where $\xi_M$ denotes the
induced vector field.
\item [b)] $H+\alpha$ is an equivariantly closed three form in the
usual Cartan Model.
\end{itemize}
\end{definition}


Therefore given a Hamiltonian action on an $H$-twisted
generalized complex manifold, there is an equivariantly closed three
form $H_G:=H+\alpha$ which is part of
the data defining the Hamiltonian action.


Next we recall the definition of the generalized equivariant cohomology
as introduced in \cite{Lin06} and compare it with the twisted equivariant cohomology $H_G(M,H_G)$.
First observe that since $H$ is an invariant three form, the operators $d_H=d-H\wedge$,
$\overline{\partial}$, and $\partial$ are $G$-equivariant and thus
have natural extensions to the space of the usual equivariant
differential forms, i.e., the Cartan model $\Omega_G$. For brevity,
we will also denote these extensions by $d_H$, $\overline{\partial}$
and $\partial$. To encode the information given by the moment one
form, \cite{Lin06} introduced two equivariant
differentials $D_G=d_H+\mathcal{A}$ and
$\overline{\partial}_G=\overline{\partial}+\mathcal{A}$ in the
Cartan model, where $\mathscr{A}$ is defined by
\begin{equation} \label{horizontal-differential}
(\mathcal{A}\gamma)(\xi)=-\iota_{\xi_M}\gamma(\xi)+\sqrt{-1}(d\mu^{\xi}+\sqrt{-1}\alpha^{\xi})\wedge
\gamma(\xi),\,\,\, \xi \in \frak{g}, \,\,\,\gamma \in
\Omega_G.\end{equation} It is easy to check that $D_G^2=0$ and
$\overline{\partial}_G^2=0$. In particular, we have the following definition
for the generalized equivariant cohomology.
\begin{definition}
 Let $\Omega_G=(S\g^* \otimes \Omega(M))^G$ be $Z_2$ graded. Then
$D_G=d_H+\mathscr{A}$ is a differential of degree $1$. And the $Z_2$
graded \textbf{generalized equivariant cohomology} is defined to be
\[ H^{\text{even/odd}}(\Omega_G,
D_G)=\dfrac{\text{ker}\left(\Omega_G^{\text{even/odd}} \xrightarrow{
D_G}
\Omega_G^{\text{odd/even}}\right)}{\text{im}\left(\Omega_G^{\text{odd/even}}
\xrightarrow{ D_G} \Omega_G^{\text{even/odd}}\right)}.\]
\end{definition}

Note that the differential operator $D_G$ on $\Omega_G$ extends naturally to a differential operator
on $\widehat{\Omega}_G$, which we will also denote by $D_G$. Define
\[H(\widehat{\Omega}_G,
D_G)=\dfrac{\text{ker}\left(\widehat{\Omega}_G \xrightarrow{
D_G}
\widehat{\Omega}_G\right)}{\text{im}\left(\widehat{\Omega}_G
\xrightarrow{ D_G} \widehat{\Omega}_G\right)}.
\]
By the discussion in Section \ref{Twisted equivariant cohomology},
it is clear that the differential operator $D_G$ on $\widehat{\Omega}_G$ is nothing
but the the usual equivariant differential $d_G$ twisted by the equivariantly closed
three form $H_G-\newi d_G \mu$. (Note that $d\mu=d_G\mu$ is an exact equivariant three form.) The
following result is a simple consequence of Lemma
\ref{basic-facts-twisted-equivariant-cohomology}.

\begin{proposition} \label{equivalence }Consider the Hamiltonian action of a compact
connected Lie group $G$ on an $H$-twisted generalized complex
manifold with generalized moment map $\mu:M\rightarrow \frak{g}^*$.
Then the map
\[H_G(M,H_G)\rightarrow H(\widehat{\Omega}_G,
D_G), \,\,\, [\gamma]\mapsto [\exp(-\newi \mu)\gamma]\] defines an
isomorphism between the $H_G$-twisted equivariant cohomology
$H_G(M,H_G)$ and the generalized equivariant cohomology.
\end{proposition}

Recall that the following result is proved in \cite{Lin06}, which is
a direct generalization of an immediate consequence of
the $d_G\delta$-lemma \cite{LS03} in symplectic geometry.

\begin{theorem}\label{ddbarlemma} If we assume that the $H$-twisted generalized complex manifold $M$ satisfies
the $\bar{\partial}\partial$-lemma, and assume the action of a
compact connected Lie group $G$ on $M$ is Hamiltonian, then as $(S\frak{g}^*)^G$-modules, $H(\Omega_G, D_G)\cong (S\g^*)^G\otimes
H(M, H)$. \end{theorem}



 Using an argument similar to the one we give to prove Proposition \ref{completion-of-cohomology},
 it is easy to see that there is a natural $(\widehat{S\g^*})^G$-module isomorphism from
 $H(\widehat{\Omega}_G,D_G)$ to $(\widehat{S\g^*})^G \otimes H(M,H)$
 by Theorem \ref{ddbarlemma}. Combining this fact with Proposition \ref{equivalence }, we have the following result.

\begin{theorem} (\textbf{Equivariant Formality})\label{Equivariant Formality}
Assume the $H$-twisted generalized complex manifold $M$ satisfies
the $\bar{\partial}\partial$-lemma, and assume the action of a
compact connected Lie group $G$ on $M$ is Hamiltonian. Let $\alpha$
be the moment one form of the Hamiltonian action and let
$H_G=H+\alpha$. Then as $(\widehat{S\frak{g}^*})^G$-modules we have that
\begin{equation} \label{canonical-iso} H_G(M,H_G)\cong
(\widehat{S\g^*})^G\otimes H(M, H),\end{equation} where
 the isomorphism depends only on the generalized
moment map.
\end{theorem}

It is interesting to give a more concrete description of the
isomorphism (\ref{canonical-iso}). Note that via the isomorphism
(\ref{canonical-iso}) one gets a canonical map
\[ s: H(M,H)\rightarrow H_G(M,H_G)\] such that the
inverse of the canonical isomorphism (\ref{canonical-iso}) is given
by  $ 1\otimes s$. Let us show how in principle one can compute the
map $s$.

 Let $\varphi$ be a $d_H$-closed differential form representing a cohomology class in
$H(M,H)$. Since the generalized complex manifold $M$ satisfies the
$\overline{\partial}\partial$-lemma, by Proposition
\ref{consequence-ddbar-lemma} we may assume that $\varphi$ is both
$\overline{\partial}$ and $\partial$ closed.  We claim that we may
assume that $\varphi$ is $G$-invariant as well. Resorting to the
usual averaging trick, to establish the claim one needs only to
check that the induced action of $G$ on $H(M,H)$ is trivial. Since
$G$ is connected, it suffices to show that all operators $L_{\xi_M}$
act trivially on $H(M,H)$, where $\xi_M$ denotes the vector field on
$M$ induced by $\xi \in \frak{g}$. However, since $\iota_{\xi_M}
H=d\alpha^{\xi}$, we have
\[\begin{split} L_{\xi_M}&=d\iota_{\xi_M}+\iota_{\xi_M}d =(d-H\wedge)(\iota_{\xi_M}+\alpha^{\xi}\wedge)+(\iota_{\xi_M}+\alpha^{\xi}\wedge)(d-H\wedge)
\\&=d_H(\iota_{\xi_M}+\alpha^{\xi}\wedge)+(\iota_{\xi_M}+\alpha^{\xi}\wedge)d_H \end{split}.\]
 That is to say that $L_{\xi_M}$ is chain homotopic to zero in
 $H(M, H)$ and our claim is established. In view of the grading (\ref{alternative-grading}), write $\varphi=\sum_i \varphi^i$, where $\varphi^i \in \Gamma(U^{i})$. Then each homogenous component
$\varphi^i$ is $G$-invariant and is both $\overline{\partial}$ and
$\partial$ closed.

Observe the grading (\ref{alternative-grading}) on the
space of differential forms induces an alternative filtration on the
Cartan model
\begin{equation}\label{filtration} F^k\Omega_G=\bigoplus_{j \leq
k}(S\g^* \otimes \Gamma(U^j))^G.\end{equation}

 Now fix an $i$ for the moment. It follows from \cite[Lemma 4.9]{Lin06} that
$\mathcal{A}\varphi^i$ is $\partial$-exact. Since
$\overline{\partial}$ anti-commutes with $\mathcal{A}$,
$\mathcal{A}\varphi^i \in F^{i-1}\Omega_G $ is both $\partial$-exact
and $\overline{\partial}$-closed. It then follows from the
$\overline{\partial}\partial$-lemma that there exists
$\gamma^{i-1}\in F^{i-1}\Omega_G$ such that
$\mathcal{A}\varphi^i=\overline{\partial}\partial \gamma^{i-1}$. (
We refer to \cite[Lemma 4.9]{Lin06} for a detailed proof of this
fact.) Also $\mathcal{A}\partial
\gamma^{i-1}=-\partial\mathcal{A}\gamma^{i-1}$ is both $\partial$
exact and $\overline{\partial}$-closed. Another application of the
$\overline{\partial}\partial$-lemma implies that there exists
$\gamma^{i-3}\in F^{i-3}\Omega_G$ such that
$\mathcal{A}\partial\gamma^{i-1}=\overline{\partial}\partial\gamma^{i-3}$.
Since the filtration (\ref{filtration}) is finite, applying this
argument recursively we get a finite chain $\gamma^k \in
F^{i+1-2k}\Omega_G$ such that
\[ \varphi^i_g:= \varphi^i+\sum_k \partial \gamma^k\] is
$\overline{\partial}_G=(\overline{\partial}+\mathcal{A})$-closed.
Note that by construction $\varphi_g^i$ is also $\partial$-closed.
So $\varphi^i_g$ is $D_G=(\overline{\partial}_G+\partial)$-closed.
Thus by Proposition \ref{equivalence } $e^{\sqrt{-1}\mu}(\varphi^i_g
)$ is $d_{G,H_G}$-closed in $\widehat{\Omega}_G$, the completion of
$\Omega_G$. Set $\varphi_g:=\sum_i \varphi^i_g$. Then
$e^{\sqrt{-1}\mu}\varphi_g$ is a $d_{G,H_G}$-closed element in $\widehat{\Omega}_G$, representing a cohomology class in $H_G(M,
H_G)$. A careful
reading of the proof of Theorem \ref{ddbarlemma} given in
\cite{Lin06} shows that $s[\varphi]=[\varphi_g]$.

 Recall that an equivariant differential form
$\eta \in \Omega_G(M)$ can be regarded as a polynomial mapping
\[ \frak{g}\rightarrow \Omega(M).\]
 Thus there is a natural map
\[p: \Omega_G(M) \rightarrow \Omega(M),\,\,\, \eta \mapsto
\eta(0),\] where $\eta (0)$ is the value of the polynomial map $\eta
:\frak{g} \rightarrow \Omega(M)$ evaluated at the zero element in
$\frak{g}$. It is easy to see that $p$ is a chain map between the
differential complexes $(\widehat{\Omega}_G(M), d_{G,H_G})$ and $(\Omega(M),
d_H)$ and so induces a map
\[ \widetilde{p}:H_G(M, H_G)\rightarrow H(M,H);\]
furthermore, by the above description of the map $s$, one checks
easily that $\widetilde{p}s=\text{id}$. In particular, we have
proved the following result.

\begin{corollary}\label{section} The canonical map $s$ is a section of
$\widetilde{p}$. Thus every cohomology class in $H(M,H)$ has a
canonical equivariant extension in $H_G(M, H_G)$.
\end{corollary}

Recall from \cite{LT05} that the action of a compact connected Lie
group $G$ on an $H$-twisted generalized K\"ahler manifold
$(M,\J_1,\J_2)$ is said to be Hamiltonian if the action preserves
the generalized K\"ahler pair $(\J_1,\J_2)$ and if there exist a
generalized moment map and a moment one-form for the twisted
generalized complex structure $\J_1$. By Theorem \ref{Gua'slemma} an
$H$-twisted generalized K\"ahler manifold always satisfies the
$\overline{\partial}\partial$-lemma with respect to both $\J_1$ and
$\J_2$. The following result is a direct consequence of Theorem
\ref{Equivariant Formality}.

\begin{theorem}\label{formality-g-kahler} Suppose that a compact connected
Lie group $G$ acts on an $H$-twisted generalized K\"ahler manifold
$(M,\J_1,\J_2)$ in a Hamiltonian fashion. Let $\alpha$ be the moment
one form and $H_G=H+\alpha$ the equivariantly closed three form.
Then we have \[ H_G(M,H_G)\cong (\widehat{S\g^*})^G\otimes H(M, H),\] where the isomorphism
is an isomorphism of $(\widehat{S\frak{g}^*})^G$-modules and depends  only on
the generalized moment map.
\end{theorem}

Next we extend the Kirwan injectivity theorem \cite{Kir86} to the
case of Hamiltonian actions on $H$-twisted generalized  K\"ahler
manifolds. We need the following intermediate result.

\begin{lemma}\label{fixed-points-set} Suppose that a compact connected torus $T$ acts
on a generalized complex manifold $(M,\J)$ in a Hamiltonian fashion.
And suppose that $W$ is a connected component of the fixed points
set $M^T$. Then $W$ is a generalized complex manifold itself and the
induced trivial $T$-action on $W$ is Hamiltonian.
\end{lemma}

\begin{proof} Let $L$ be the $\sqrt{-1}$ eigenbundle of
$\J: T_{\C}M\oplus T^*_{\C}M\rightarrow T_{\C}M\oplus T^*_{\C}M$. It
is proved in \cite[Lemma 5.4]{Lin06} that $X$ carries a generalized
complex structure whose $\sqrt{-1}$ eigenbudle is
\[L_{W,x}=\{X+(\xi \mid_{T_{\C}W})\,\,\vert \,\, X+\xi \in
L\cap\left(T_{\C,x}W\oplus T_{\C,x}^*M\right)\}.\] Let
$\mu:M\rightarrow \frak{t}^*$ and $\alpha\in \Omega^1(M,\frak{t}^*)$
be the generalized moment map and the moment one form of the $T$
action on $M$ respectively. Let $\mu_W$  and $\alpha_W$ be the
restrictions of $\mu$ and $\alpha$ to $W$.  Then it is easy to check
that $\sqrt{-1}(d\mu^{\xi}_W+\sqrt{-1}\alpha^{\xi}_W) \in
C^{\infty}(L_W)$ for any $\xi \in \frak{t}$. This proves that the
trivial $T$ action on $W$ is Hamiltonian.
\end{proof}

\begin{theorem}({\bf Kirwan injectivity}) Suppose that a compact connected
Lie group $G$ acts on an $H$-twisted generalized K\"ahler manifold
$(M,\J_1,\J_2)$ in a Hamiltonian fashion. Then the canonical map
\[ i^*: H_{T}^*(M,H_T) \rightarrow H_{T}^*(M^T,i^*H_T),\]
induced by the inclusion $i:M^T\rightarrow M$ is injective, where
$M^T$ is the fixed points set of the torus $T$ action.
\end{theorem}

\begin{proof} It is proved in \cite[Prop. 5.5]{Lin06} any connected component of the fixed
points set is a generalized K\"ahler manifold itself. Combining this
fact with Lemma \ref{fixed-points-set} and Theorem
\ref{formality-g-kahler}, one sees immediately that
$H_{T}(M^T,i^*H_T)$ is a free $\widehat{S\frak{t}^*}$-module. Since $S\t^*$ is a subring of $\widehat{S\frak{t}^*}$,
$H_{T}(M^T,i^*H_T)$ must be a free $S\t^*$-module as well. So if $[\alpha]$
is a torsion element in $H_T(M,H_T)$, it gets mapped to zero in
$H_{T}^*(M^T,i^*H_T)$. On the other hand, Theorem
\ref{t--localizationI} implies that the kernel of the map $i^*$ is a
torsion module. So the kernel of $i^*$ is the module of torsion
elements in $H_T(M,H_T)$. However, by Theorem
\ref{formality-g-kahler} $H_T(M,H_T)$ is a free $\widehat{S\frak{t}^*}$-module. So it must be a free $S\frak{t}^*$-module
and it does not have any non-trivial torsion elements. We conclude that the
map $i^*$ is injective.
\end{proof}


\section{Some reflections on the cohomology of generalized complex
quotients}\label{generalized-complex-quotients}

Let a compact Lie group $G$ act on a twisted generalized complex
manifold $(M,\J)$ with generalized moment map $\mu$. Let $\cO_a$ be
the co-adjoint orbit through $a \in \g^*$. If $G$ acts freely on
$\mu^{-1}(\cO_a)$, then $\cO_a$ consists of regular values and $M_a
= \mu^{-1}(\cO_a)/G$ is a manifold, which is called the {\bf
generalized complex quotient}. The following  two results were
proved in \cite{LT05}.

\begin{lemma}\label{twist three form}
Let a compact Lie group $G$ act freely on a manifold $M$. Let $H$ be
an invariant closed three form and let $\alpha$ be an equivariant
mapping from $\g$ to $\Omega^1(M)$.  Fix a connection $\theta \in
\Omega(M,\g^*)$. Then if $H + \alpha \in \Omega^3_G(M)$ is
equivariantly closed, there exists a natural form $\Gamma \in
\Omega^2(M)^G$ so that $\iota_{\xi_M} \Gamma = \alpha^\xi$. Thus $H
+ \alpha + d_G \Gamma \in \Omega^3(M)^G \subset \Omega_G^3(M)$ is
closed and basic and so descends to a closed form $\widetilde{H} \in
\Omega^3(M/G)$ so that $[\widetilde{H}]$ is the image of $[H  +
\alpha]$ under the Kirwan map.
\end{lemma}

\begin{proposition} \label{Twisted Complex Reduction}
Assume there is a Hamiltonian action of a compact Lie group $G$ on
an $H$-twisted generalized complex manifold $(M,\mathcal{J})$ with
generalized moment map $\mu: M \to \g^*$ and moment one-form $\alpha
\in \Omega^1(M,\g^*)$. Let $\cO_a$ be a co-adjoint orbit through $a
\in \g^*$ so that $G$ acts freely on $\mu^{-1}(\cO_a)$. Given a
connection on $\mu^{-1}(\cO_a)$, the generalized complex quotient
$M_a$ inherits an $\widetilde{H}$-twisted generalized complex
structure $\widetilde{\J}$, where $\widetilde{H}$ is defined as in
the Lemma \ref{twist three form}. Up to $B$-transform,
$\widetilde{J}$ is independent of the choice of connection. Finally,
for all $m \in \mu^{-1}(\cO_a)$,
$$\type(\widetilde{\mathcal{J}})_{[m]} = \type(\mathcal{J})_m.$$
\end{proposition}

It is an interesting question if we can read off the cohomology of
the generalized complex quotient $M_a$ from the equivariant
cohomology of the Hamiltonian generalized complex manifold $M$. The
Kirwan map is the key to understanding this question.

Assume that $G$ acts freely on $\mu^{-1}(\cO_a)$. Let
$i:\mu^{-1}(\cO_a)\rightarrow M$ be the inclusion map and $\pi:
\mu^{-1}(\cO_a)\rightarrow M_a$ the quotient map. Then there is a
Kirwan map $ k: H_G(M)\rightarrow H(M_a)$ such that the following
diagram commutes:
\begin{diagram}
 H_G(M) &\rTo^{i^*}      & H_G(\mu^{-1}(\cO_a)) \\
 &\rdTo^{k}         &\dTo^{\cong}  &            \\
  & & H(M_a).
\end{diagram}
Here the vertical map is given by the inverse of the map \[\pi^*:
H(M_a) \rightarrow H_G(\mu^{-1}(\cO_a)),\,\,\,[\eta]\mapsto
[\pi^*\eta],\] which is an isomorphism due to Theorem
\ref{Cartan--Map}.

If the generalized complex structure $\J$ on $M$ is induced by a
symplectic structure, then it follows from the famous Kirwan
surjectivity theorem  \cite{Kir86} that the Kirwan map is
surjective. Let us give a very short explanation why the Kirwan
surjectivity theorem holds in the symplectic case. For simplicity,
let us only consider $S^1$ action here. In this case, the fixed
points set of the action are exactly the critical points set of the
moment map and the moment map is a Morse-Bott function. Besides, the
norm square of the moment map induces a Morse stratification on the
manifold $M$ from which one derives the Kirwan surjectivity theorem.
However, for Hamiltonian $S^1$ action on general complex manifolds,
in general the generalized moment map is not a Morse-Bott function.
Indeed, let $\mu: M\rightarrow \R$ be the generalized moment map and
$\alpha$ the moment one form. Then the condition $ -X
+\sqrt{-1}(d\mu+\sqrt{-1}\alpha) \in C^{\infty}(L)$ is equivalent to
that
\[Jd\mu=-X -\alpha,\] where $ X$ is the vector field generated by the
$S^1$ action, and $L$ is the $\sqrt{-1}$ eigenbundle of the
generalized complex structure. Let $\text{Crit}(\mu)$ be the set of
critical points of the moment map $\mu$ and $M^{S^1}$ the fixed
points set of the $S^1$ action. It is easy to see that
\[ \text{Crit}(\mu) \subset M^{S^1}.\]
 However, in general $\text{Crit}(\mu) \neq M^{S^1}$. For instance, in
\cite[Section 5]{Hu05} it was shown that there is a generalized
K\"ahler structure $(\J_1,\J_2)$ on $\C^2\setminus \{0\}$ and a
Hamiltonian $S^1$ action on $(\C^2\setminus \{0\},\J_1,\J_2)$ such that
for any regular value $a$ of the generalized moment map $\mu$, the
level set $\mu^{-1}(a)$ contains a fixed points submanifold.  It is
unclear to the author at the moment whether the Kirwan surjectivity
still holds or not for Hamiltonian actions on twisted generalized
complex manifolds in general.

Nevertheless, the image of the Kirwan map forms a subgroup of the
cohomology group of the generalized complex quotient which may
contain many interesting cohomology classes. For instance, it is an
interesting question when the generalized complex quotient has a
non-trivial twisting. Observe that by Lemma \ref{twist three form}
and Proposition \ref{Twisted Complex Reduction} the cohomology class
of the twisting three form $\widetilde{H}$ on $M_a$ is  exactly the
image of the equivariant cohomology class $[H+\alpha]$ under the
Kirwan map $k:H_G(M)\rightarrow H(M_a)$. This innocuous observation
leads to the following result.

\begin{proposition} Assume that $(M, \J)$ is an $H$-twisted generalized complex
manifold such that $H^{\text{odd}}(M)=0$. And assume there is a
Hamiltonian action of a compact connected Lie group $G$ on $M$ with
generalized moment map $\mu$ and moment one form $\alpha$. Let
$\cO_a$ be a co-adjoint orbit through $a \in \g^*$ so that $G$ acts
freely on $\mu^{-1}(\cO_a)$. Then the twisted three form
$\widetilde{H}$ on $M_a=\mu^{-1}(\cO_a)/G$, as defined in Lemma
\ref{twist three form}, is cohomologically trivial.
\end{proposition}

\begin{proof} It is  shown in \cite[Thm. 6.5.3]{GS99} that if $H^{\text{odd}}(M)=0$ then the
spectral sequence of the Cartan complex of the $G$-manifold $M$
degenerates at the $E_1$ stage. In particular, we have that
$H^{\text{odd}}_G(M) \cong \bigoplus_{p+q=\text{odd}}
E_1^{p,q}=H^{\text{odd}}(M)\otimes (S\frak{g}^*)^G=0$. This shows
clearly that $H_G(M)$ contains no non-trivial odd dimensional
equivariant cohomology classes. We conclude that $[H+\alpha]=0$ and
so $[\widetilde{H}]=k([H+\alpha])=0$.

\end{proof}

\begin{remark} Physicists are interested in producing explicit examples of gauged
sigma models. When there is no flux, i.e., the target manifold does
not have a cohomologically non-trivial twisting, a very useful tool
producing such examples is to use the so called gauged linear model
which corresponds to the K\"ahler quotient of a K\"ahler vector
space. Naturally, physicists had hoped that one can produce explicit
examples of generalized K\"ahler quotients of generalized K\"ahler
vector spaces such that the quotient spaces have a cohomologically
non-trivial twisting. The above result asserts that this is not
possible since $H^{\text{odd}}(V)=0$ for any vector space $V$.
\end{remark}

We finish this section by showing that in the framework of the
twisted equivariant cohomology there is also a well-defined Kirwan
map. This strongly suggests that the results we obtained in Section
\ref{t-eq-theory-t-complex} do help us understand the cohomology of
the generalized complex quotient. To describe the Kirwan map
in this setting, we will need Proposition \ref{twisted-version-extended-Cartan-theorem}
in the Appendix \ref{generalized-eq-diff-forms}.




\begin{proposition}\label{twisted-Kirwan-map} Assume a compact connected Lie group
$G$ acts on an $H$-twisted generalized complex manifold $M$ with
generalized moment map $\mu:M \rightarrow \frak{g}^*$ and moment one
form $\alpha \in \Omega^1(M,\frak{g}^*)$. Assume that $G$ acts
freely on the level set $\mu^{-1}(0)$. Then there is a Kirwan map $
k: H_G(M,H+\alpha) \rightarrow H(M_0,\widetilde{H})$ such that the
following diagram commutes:
\begin{diagram}
 H_G(M, H+\alpha) &\rTo^{i^*}      & H_G(\mu^{-1}(0),i^*(H+\alpha) ) \\
 &\rdTo^{k}         &\dTo^{\cong}  &            \\
  & & H(M_0,\widetilde{H}),
\end{diagram}
where $i^*: H_G(M, H+\alpha) \rightarrow
H_G(\mu^{-1}(0),i^*(H+\alpha))$ is induced by the inclusion map
$i:\mu^{-1}(0)\rightarrow M$, $M_0=\mu^{-1}(0)/G$ is the generalized
complex quotient taken at the zero level of the generalized moment
map, and $\widetilde{H}$ is the twisting three form on $M_0$
obtained as in Lemma \ref{twist three form} and Proposition
\ref{Twisted Complex Reduction}.

\end{proposition}

\begin{proof} By Lemma \ref{twist three form} and Proposition \ref{Twisted Complex
Reduction}, we can find a closed three form $\eta \in
\Omega^3(\mu^{-1}(0))$ such that $[\eta]=[i^*(H+\alpha)]$ and such
that $\eta =\pi^*\widetilde{H}$. Since $\eta$ and $i^*(H+\alpha)$
represents the same cohomology class, we have
\[H_G(\mu^{-1}(0),\eta ) \cong H_G(\mu^{-1}(0),i^*(H+\alpha) ).\]
Corollary \ref{twisted-Kirwan-map} now follows easily from
Proposition \ref{twisted-version-extended-Cartan-theorem}.
\end{proof}

\section{The Duistermaat-Heckman theorem for generalized Calabi-Yau
manifolds}\label{DH-theorem-for-g-Calabi-Yau}

Consider the Hamiltonian action of a compact connected Lie group $G$
on an $H$-twisted generalized Calabi-Yau manifold $M$. Our first
observation is that if the action is free at a level set of the
generalized moment map, then the generalized complex quotient taken
at the level set inherits a generalized Calabi-Yau structure. The
same result, under the assumption that both the twisting three form
$H$ and the moment one form vanish, has already been proved by Nitta
\cite{NY06}.  We will make use of the following fact established in
the proof of \cite[Thm. A]{NY06}.

\begin{lemma} \label{restriction-of-g-calabi-yau}Suppose that there
is a Hamiltonian action of a compact connected Lie group $G$ on an
$H$-twisted generalized Calabi-Yau manifold $M$ with generalized
moment map $\mu:M \rightarrow \frak{g}^*$ and zero moment one form.
Then the restriction of the generalized Calabi-Yau structure $\rho$
to $\mu^{-1}(0)$ is nowhere vanishing.
\end{lemma}

\begin{remark} Although \cite{NY06} assume the
generalized Calabi-Yau structure is untwisted, the proof of the
above fact given by Nitta uses only the linear algebra of generalized
complex structures and so applies to the twisted case as well.

\end{remark}

\begin{proposition} Assume that there is a Hamiltonian action of a
compact connected Lie group $G$ on an $H$-twisted generalized
Calabi-Yau manifold $(M,\J)$ with generalized moment map
$\mu:M\rightarrow \frak{g}^*$ and moment one form $\alpha \in
\Omega^1(M,\frak{g}^*)$, preserving the generalized Calabi-Yau
structure $\rho$. And assume   $\cO_a$ is a co-adjoint orbit through
$a \in \g^*$ so that $G$ acts freely on $\mu^{-1}(\cO_a)$. Then
given a connection on $\mu^{-1}(\cO_a)$, the generalized complex
quotient $M_a$ inherits an $\widetilde{H}$-twisted generalized
Calabi-Yau structure $\widetilde{\rho}$, such that
$\pi^*\widetilde{\rho}=\rho \mid_{\mu^{-1}(\cO_a)}$, where $\pi:
\mu^{-1}(\cO_a) \rightarrow M_a$ is the quotient map,
$\widetilde{H}$ is defined as in Lemma \ref{twist three form}. Up to
$B$-transform, $\widetilde{\rho}$ is independent of the choice of
connections. Finally, for all $m \in \mu^{-1}(\cO_a)$,
$$\type(\widetilde{\mathcal{J}})_{[m]} = \type(\mathcal{J})_m,$$
where $\widetilde{J}$ is the quotient generalized complex structure
on $M_a$.
\end{proposition}

\begin{proof} It suffices to show that the restriction of $\rho$ to
$\mu^{-1}(\cO_a)$ descends to a generalized Calabi-Yau structure
$\widetilde{\rho}$ on $M_a$ such that $\pi^*\widetilde{\rho}=\rho
\mid_{\mu^{-1}(\cO_a)}$. Using the argument given in the proof of
\cite[Prop. 3.8, A.7]{LT05}, without the loss of generality, we can
assume that the coadjoint orbit $\cO_a=0$, and assume that in an
open neighborhood of the level set $ \mu^{-1}(0)$,  the moment one
form is zero and the twisting three form $H$ is basic so that it
descends to the twisting three form $\widetilde{H}$ on the quotient
$M_0:=\mu^{-1}(0)/G$. By the definition of the generalized moment
map, we have that for any $\xi \in \frak{g},
\xi_M-\sqrt{-1}d\mu^{\xi} \in C^{\infty}(L)$, where $L$ is the
$\sqrt{-1}$ eigenbundle of the generalized complex structure $\J$.
By the definition of a generalized Calabi-Yau structure, we have
\[(\xi_M -\sqrt{-1}d\mu^{\xi})\cdot
\rho=\iota_{\xi_M}\rho-\sqrt{-1}d\mu^{\xi}\wedge \rho=0.\] It
follows that $\iota_{\xi_M}\rho
\mid_{\mu^{-1}(0)}=\left(\sqrt{-1}d\mu^{\xi}\wedge \rho\right)
\mid_{\mu^{-1}(0)}=0$ for any $\xi \in \frak{g} $. Therefore $\rho
\mid_{f^{-1}(0)}$ is a basic form and so descends to a form
$\widetilde{\rho}$ such that $\pi^*\widetilde{\rho}=\rho
\mid_{\mu^{-1}(0)}$. Since $\rho$ is $d_H$-closed,
$\widetilde{\rho}$ is $d_{\widetilde{H}}$-closed; moreover, since by
Lemma \ref{restriction-of-g-calabi-yau} $\rho \vert_{\mu^{-1}(0)}$
is nowhere vanishing, $\widetilde{\rho}$ is nowhere vanishing as
well. Thus
\[L_{\widetilde{\rho}}(x)=\{A \in T_{\C,x}M_0\oplus T^*_{\C,x}M_0
\,\vert\, A \cdot \widetilde{\rho} =0\} \] is an isotropic subspace
of $T_{\C,x}M_0\oplus T_{\C,x}^*M_0$, $x\in M_0$.  Finally, let
$\widetilde{L}$ be the $\sqrt{-1}$ eigenbundle of the quotient
generalized complex structure on $M_0$. Using the description of
$\widetilde{L}$ in \cite[Prop. 3.8]{LT05} it is straightforward
to check that $\widetilde{\rho}$ is annihilated by the Clifford
action of the sections of $\widetilde{L}$. Hence we have
$\widetilde{L}_x\subset L_{\widetilde{\rho}}(x)$,  $x\in M_0$.
However, $L_x$ is a maximal isotropic subspace of $T_{\C,x}M_0\oplus
T^*_{\C,x}M_0$. So we must have
$\widetilde{L}_x=L_{\widetilde{\rho}}(x)$, $x \in M$. In other
words, $\widetilde{\rho}$ is the pure spinor associated to the
quotient generalized complex structure $\widetilde{J}$ on $M_0$.
This finishes the proof that $\widetilde{\rho}$ is a generalized
Calabi-Yau structure.
\end{proof}

\begin{example}\label{Hamiltonian-symplectic} Let $G$ act on a symplectic manifold $(M, \omega)$ with moment map
$\Phi \colon M \to \g^*$, that is, $\Phi$ is equivariant and $
\iota_{\xi_M} \omega = - d\Phi^\xi$ for all $\xi \in \g$. Then the
action of $G$ also preserves the generalized complex structure
$\mathcal{J}_\omega$, where $\J_{\omega}$ is defined as in Example
\ref{symplectic-g-calabi-Yau},  and it is straightforward to check
that the action on the generalized complex manifold
$(M,\J_{\omega})$ is Hamiltonian with generalized moment map $\Phi$
and trivial moment one form. By Example
\ref{symplectic-g-calabi-Yau} $\omega$ defines a generalized
Calabi-Yau structure $e^{i\omega}$. Suppose that $G$ acts freely on
the level set $\mu^{-1}(0)$. Then the restriction of $e^{i\omega}$
to the level set $\mu^{-1}(0)$ descends to the quotient generalized
Calabi-Yau structure $e^{i\omega_0}$ on $M_0=\mu^{-1}(0)/G$, where
$\omega_0$ is the reduced symplectic form on $M_0$.

\end{example}

 On the generalized complex quotient $M_a$, the quotient
generalized Calabi-Yau structure $\widetilde{\rho}_a$ depends on the
choice of a connection, c.f., \cite[Prop. A.7]{LT05}. However,
a different choice of the connection results in a new quotient
generalized Calabi-Yau structure which is given by
$e^{-B}\widetilde{\rho}_a$ for some closed two form $B \in
\Omega^2(M_a)$. It follows immediately from Lemma
\ref{invariance-mukai-pairing} that, on the quotient $M_a$,
$(\widetilde{\rho}_a,\overline{\widetilde{\rho}}_a)$ is
well-defined, independent of the choice of connections, where
$(\cdot,\cdot)$ denotes the Mukai pairing on the quotient space
$M_a$.

From now on we consider the Hamiltonian action of a compact connected torus $T$ on
an $H$-twisted generalized complex manifold $M$ with a proper generalized moment map
$\mu$. Assume that $a_0\in \frak{t}^*$ such that
the action of $T$ on the level set $\mu^{-1}(a_0)$ is free. Since by
Lemma \ref{free-orbit} below the set on which $T$ acts freely is
open and since the generalized moment map is proper, there is an
open subset $U\subset \frak{t}^*$ such that for any $a \in U$ the
action of $T$ on $\mu^{-1}(a)$ is free.  We define the {\bf
Duistermatt-Heckman function} $f$ on $U \subset \frak{t}^*$ by the
following formula.
\begin{equation} \label{DHfunction} f(a)=\dfrac{(-1)^{n+\frac{k(k+1)}{2}}(2\pi)^k}{(2i)^{n-k}} \int_{M_a}
(\widetilde{\rho}_a,\overline{\widetilde{\rho}}_a),\end{equation}
where $\widetilde{\rho}_a$ is the quotient generalized Calabi-Yau
structure on $M_a$, and $M_a$ is given the orientation induced by
$(\widetilde{\rho}_a,\overline{\widetilde{\rho}}_a)$.

\begin{lemma}\label{free-orbit}(\cite{GS84})  Assume that a torus $T$ acts on a connected manifold $M$
 effectively. Then the set $M'$ on which $T$ acts freely is equal to
 the complement of a locally finite union of closed
 manifolds of codimension $\geq 2$. In particular, $M'$ is open,
 connected, dense, and $M\setminus M'$ has measure zero.
 \end{lemma}

\begin{theorem}\label{push-forward} Suppose that a $k$ dimensional torus $T$ acts on
a $2n$ dimensional connected generalized Calabi-Yau manifold $M$
freely such that the action is Hamiltonian with a proper generalized
moment map $\mu:M\rightarrow \t^*$ and a moment one form $\alpha \in
\Omega^1(M,\t^*)$, and such that the action preserves the
generalized Calabi-Yau structure $\rho$ on $M$.  Then the density
function of the push-forward measure $\mu_*(dm)$ coincides with the
Duistermaat-Heckman function $f$ defined by (\ref{DHfunction}),
where $dm$ is the measure on $M$ defined by the volume form
(\ref{volume-form}) on $M$.
\end{theorem}

\begin{proof}  Since the torus action is free, after applying a $B$-transform,
we may well assume that the moment one form $\alpha$ is trivial.
Choose an integer lattice in $\frak{t}=\text{Lie}(T)$, and identify
$T$ with $S^1\times \underbrace{ S^1\cdots \times S^1}_{k-1}=S^1
\times T^{k-1}$. Following the proof of \cite[Thm. 5.8]{GGK97}, we
consider an open neighborhood of a free orbit in $M$. On such a
neighborhood there exists a coordinate system
\[\theta_1,
\theta_2,\cdots,\theta_k,x_1,\cdots,x_k,y_1,\cdots,y_{2d},\] where
$2d=2n-2k, \theta_i \in \R/{2\pi \Z}$, and $x_i$ are coordinates on
$\frak{t}^*$, such that the $T$ action is generated by the vector
field
$\frac{\partial}{\partial\theta_1},\frac{\partial}{\partial\theta_2},
\cdots,\frac{\partial}{\partial\theta_k}$ and the moment map $\mu$
is given by $(x_1,x_2,\cdots,x_k)$. We refer to the proof of
\cite[Thm. 5.8]{ GGK97} for a detailed explanation on the existence
of such a coordinate system. Then we claim that for any
$a=(a_1,a_2,\cdots,a_k)$ in an open subset of $\frak{t}^*$ such that
$x_1,x_2,\cdots,x_k$ are well defined, we have
\begin{equation}\label{claim} \left(\rho(x), \overline{\rho}(x)\right)_M
=(-2i)^k(\theta_1\wedge dx_1 \wedge \theta_2\wedge dx_2 \wedge\cdots
\wedge\theta_k\wedge dx_k)\wedge\left(\widetilde{\rho}_a(\pi(x)),
\overline{\widetilde{\rho}}_a(\pi(x))\right)_{M_a},\end{equation}
where $\pi:\mu^{-1}(a)\rightarrow M_a$ is the quotient map, $x$ is a
point in the level set $\mu^{-1}(a)$, and $(\cdot,\cdot)_M$ and
$(\cdot,\cdot)_{M_a}$ denote the Mukai pairing on $M$ and $M_a$
respectively. We will establish our claim by induction on $k$. If
$k=1$, then since the action is Hamiltonian,
$(\frac{\partial}{\partial \theta_1}-\sqrt{-1}dx_1)\cdot \rho=0$,
i.e., $\iota_{\frac{\partial}{\partial \theta_1}}
\rho=\sqrt{-1}dx_1\wedge \rho.$  A simple calculation shows that
under the above coordinate system $\rho$ must be of the form
\[dx_1\wedge \alpha_0+i\theta_1\wedge dx_1\wedge \alpha_1 +
\alpha_1,\] where $\alpha_i$ is a differential form such that
$\iota_{\frac{\partial}{\partial \theta_1}}\alpha_i=0$ and
$\iota_{\frac{\partial}{\partial x_1}}\alpha_i=0$, $i=0,1$. It is
easy to check directly that the restriction of $\alpha_1$ to the
level set $x_1^{-1}(a)$ descends to a quotient Calabi-Yau structure
$\widetilde{\alpha}_1$ on the quotient $x_1^{-1}(a_1)/S^1$;
moreover, we have
\begin{equation}\label{step1} \left(\rho(x),
\overline{\rho}(x)\right)_M=-2i(\theta_1\wedge
dx_1)\wedge\left(\widetilde{\alpha}_1(\pi_1(x)),
\overline{\widetilde{\alpha}}_1(\pi_1(x))\right)_{M_{a_1}},\end{equation}
where $\pi_1: x_1^{-1}(a_1)\rightarrow M_{a_1}=x_1^{-1}(a_1)/S^1$ is
the quotient map.  If $k=1$, then $M_{a_1}=M_a$ and
$\widetilde{\alpha}_1$ is actually the quotient generalized
Calabi-Yau structure $\widetilde{\rho}_a$ on it. The claim is proved
for the case $k=1$. Assume that the claim is true for $k-1$ and
consider the case $k$. Note that the action of $T^{k-1} \subset
T=S^1\times T^{k-1}$ on $M$ commutes with that of $S^1$ and so
descends to a Hamiltonian action on $M_{a_1}$ with moment map
$(x_2,\cdots,x_{k})$. Here by abuse of notation, we have used
the same $x_i$ to denote the functions on the quotient $M_{a_1}$
whose pull-back under the quotient map $\pi_1$ coincides with
$x_i\mid_{x_1^{-1}(a_1)}$, $2\leq i \leq k$.

 Observe that $M_a=\left(x_2^{-1}(a_2)\cap
\cdots x_k^{-1}(a_k)\right)/T^{k-1}$ and the restriction of
$\widetilde{\alpha}_1$  to $x_2^{-1}(a_2)\cap \cdots x_k^{-1}(a_k)$
descends to the quotient generalized Calabi-Yau structure
$\widetilde{\rho}_a$ on $M_a$. Using the induction assumption, we
get that \begin{equation} \label{step2}
\left(\widetilde{\alpha}_1(y),
\overline{\widetilde{\alpha}}_1(y)\right)_{M_{a_1}}=(-2i)^{k-1}(\theta_2\wedge
dx_2 \wedge\cdots \wedge\theta_k\wedge
dx_k)\wedge\left(\widetilde{\rho}_a(\pi_2(y)),\overline{\widetilde{\rho}}_a(\pi_2(y))\right)_{M_a},\end{equation}
where $\pi_2: x_2^{-1}(a_2)\cap \cdots x_k^{-1}(a_k) \rightarrow
M_a=\left(x_2^{-1}(a_2)\cap \cdots x_k^{-1}(a_k)\right)/T^{k-1}$ is
the quotient map, $y$ is a point in the level set $x_2^{-1}(a_2)\cap
\cdots x_k^{-1}(a_k)$. Combining (\ref{step1}) and (\ref{step2}), we
conclude our claim holds for the case $k$.

To finish the proof, we borrow the following argument from
\cite{GGK97}. Let $\lambda_j$ be an invariant partition of unit such
that each $\lambda_j$ is supported in a neighborhood with
coordinates described above.  Then we have
\[\begin{split} \mu_{\ast}(dm)&
=\dfrac{(-1)^n}{(2i)^n}\sum_j \int_M  \lambda_j\left(\rho(x),
\overline{\rho}(x)\right)_M
\\&=\dfrac{(-1)^{n-k}}{(2i)^{n-k}}\sum_j \int  \lambda_j (\theta_1\wedge dx_1
 \wedge\cdots \wedge\theta_k\wedge
dx_k)\wedge\left(\widetilde{\rho}_a(\pi(x)),\overline{\widetilde{\rho}}_a(\pi(x))\right)_{M_a}
\\&= \dfrac{(-1)^{n+\frac{k(k+1)}{2}}(2\pi)^k}{(2i)^{n-k}}\sum_j \int \lambda_j\left( \int_{M_a}
\left(\widetilde{\rho}_a(\pi(x)),\overline{\widetilde{\rho}}_a(\pi(x))\right)_{M_a}\right)d_1dx_2\cdots
dx_k \,\,\,\\&(\text{ Because of Fubini's theorem})
\\&= \dfrac{(-1)^{n+\frac{k(k+1)}{2}}(2\pi)^k}{(2i)^{n-k}}\int \sum_j \lambda_j \left( \int_{M_a}
\left(\widetilde{\rho}_a(\pi(x)),\overline{\widetilde{\rho}}_a(\pi(x))\right)_{M_a}\right)dx_1dx_2\cdots
dx_k
\\&= \int f(a) dx_1dx_2\cdots
dx_k\,\,\,.
\end{split}\]

This proves Theorem \ref{push-forward}.
\end{proof}

Before we state the Duistermaat-Heckman theorem in the generalized
Calabi-Yau setting, we first give two simple observations.

\begin{lemma}\label{anti-automorphism2} Let $\sigma$ be the
anti-automorphism defined on $C^{\infty}(\wedge T^*M)$ as in
(\ref{anti-auto1}), and let $H$ be a closed thee form on the
manifold $M$. We have

\begin{itemize} \item [a)] if a differential form $\varphi$ is $d_H=(d-H\wedge)$-closed,
then $\sigma(\varphi)$ is $d_{-H}=(d+H\wedge)$-closed;
\item [b)] if $X+\gamma \in C^{\infty}(TM\oplus T^*M)$ and if $(X+\gamma)\cdot \varphi=\iota_X \gamma +\gamma \wedge
\varphi=0$, then we have $(X-\gamma)\cdot\sigma(\varphi)=\iota_X
\sigma(\varphi)-\gamma \wedge \sigma(\varphi)=0$. \end{itemize}

\end{lemma}

\begin{proof} Decompose $\varphi=\sum_{i=0}^{p}\varphi_i$  into
homogeneous components by the usual grading of differential forms,
where $p=\text{dim} M$. Then  it follows from $d_H\varphi=0$ that
$d\varphi_i=0$ for $i\leq 1$ and $d\varphi_i =H\wedge \varphi_{i-2}$
for $i \geq 2$. It is straightforward to check that
$d\sigma(\varphi_i)=0$ if $ i\leq 1$ and
$d\sigma(\varphi_i)=-H\wedge \sigma(\varphi_{i-2})$ if $ i \geq 2$.
So we have that $\sigma(\varphi)$ is $d_{-H}=(d+H\wedge)$-closed.
This proves Assertion (a) in Lemma \ref{anti-automorphism2}. A
similar argument proves Assertion (b).
\end{proof}

\begin{lemma} \label{eq-extension}Assume that there is a Hamiltonian action of a compact connected Lie group
$G$ on an $H$-twisted generalized Calabi-Yau manifold $M$ with
proper generalized moment map $\mu:M\rightarrow \frak{t}^*$ and
moment one form $\alpha \in \Omega^1(M, \frak{g}^*)$.  Then
$e^{i\mu}\rho$ is an $H_G=(H+\alpha)$-twisted equivariantly closed
form with coefficients in the ring of formal power series.
\end{lemma}

\begin{proof} For any $\xi \in \t$, we have
\[ \begin{split} \left(d_{G,H_G}
e^{i\mu}\rho\right)(\xi)&=(\sqrt{-1}du^{\xi}\wedge
\rho-\iota_{\xi_M}\rho-\alpha^{\xi}\wedge
\rho)e^{i\mu^{\xi}}+e^{i\mu}(d_H\rho)\\&=
\left(-\xi_M+\sqrt{-1}(d\mu^{\xi}+\sqrt{-1}\alpha^{\xi})\right)\cdot
\rho
\\&=0.\end{split}\]
The last equality holds because $-\xi_M
+\sqrt{-1}(d\mu^{\xi}+\sqrt{-1}\alpha^{\xi})$ is a section of the
$\sqrt{-1}$-eigenbundle of the generalized Calabi-Yau structure so
that its Clifford action annihilates $\rho$.

\end{proof}

\begin{remark}\label{why-enlarge-coefficient-ring} The same observation that $e^{i\mu}\rho$ is an equivariant closed extension of
$\rho$ has been made in \cite{HuUribe06}, and in \cite{NY07} for the
case $H=0$ and $\alpha=0$. However, $e^{i\mu}\rho$ was treated as an
equivariant differential form in the usual Cartan model in
\cite{NY07}. This is incorrect since
$e^{i\mu}\rho: \frak{g}\rightarrow \Omega(M)$ is not a polynomial
map, where $\frak{g}$ is the Lie algebra of $G$. For instance, let
$G$ act on a symplectic manifold $(M, \omega)$ with moment map $\Phi
\colon M \to \g^*$, and let $\rho=e^{i\omega}$. Then
$e^{i\Phi}e^{i\omega}=e^{i(\omega+\Phi)}=e^{i\omega_G}$, where
$\omega_G:=\omega+\Phi$ is the equivariant symplectic form. It is
well known that $e^{i\omega_G}$ is not an equivariant differential
form in the Cartan model, c.f., \cite[pp. 167]{GS99}. Therefore,
strictly speaking, the proof of the Duistermaat-Heckman theorem in
the generalized Calabi-Yau setting given in \cite{NY07} has a gap.
\end{remark}

\begin{theorem}\label{Duistermaat-Heckman}
Assume that there is a free Hamiltonian action of a $k$ dimensional
torus $T$ on an $H$-twisted $2n$ dimensional generalized Calabi-Yau
manifold $M$ with proper generalized moment map $\mu:M\rightarrow
\frak{t}^*$ and moment one form $\alpha \in \Omega^1(M,\frak{t}^*)$,
preserving the generalized Calabi-Yau structure $\rho$. Then the
density function $f$ for the push-forward measure on $\frak{t}^*$
via the generalized moment map $\mu$ is a polynomial of degree at
most $n-k$.
\end{theorem}

\begin{proof} As the action is free, after applying a
$B$-transform, one can assume that the moment one form $\alpha$ is
zero. Under this assumption,  $H_G=H$ is an equivariantly closed
three form in the usual Cartan model. First, by Lemma
\ref{eq-extension} we have $d_{G,H}(e^{i\mu}\rho)=0$. Since
$d_{G,H}$ is a real operator, we also have
$d_{G,H}(e^{-i\mu}\overline{\rho})=0$. Next we observe that the
canonical anti-automorphism $\sigma: \Omega^*(M)\rightarrow
\Omega^*(M)$ extends naturally to the equivariant differential
forms. Moreover, as an easy consequence of Lemma
\ref{anti-automorphism2}, we have that
\[ \begin{split} d_{G,-H}(e^{-i\mu}\sigma(\rho))&=e^{-i\mu}(d+H\wedge)
\sigma(\rho)+ e^{-i\mu}(-\sqrt{-1}d\mu^{\xi}-\xi_M)\cdot
\sigma(\rho)=0\end{split}\] Consequently by Lemma
\ref{module-structure} $e^{-2i\mu} \sigma(\rho)\wedge
\overline{\rho}$ is equivariantly closed under the usual (untwisted)
equivariant differential $d_G$.

Now choose a $T$-invariant $\t^*$ valued connection one form
$\theta$ for the principal bundle $M\rightarrow M/T$. Choose a basis
of $\t$ so as to identify $\t^*$ with $\R^k$ and suppose
$\theta=(\theta^1,\theta^2,\cdots,\theta^k)$,
$\mu=(\mu^1,\mu^2,\cdots,\mu^k)$ under this identification. Let $c_l
$ be the unique closed two form on $M/T$ such that $\pi^*
c_l=d\theta^l$, $1\leq l \leq k$. As we explained in Appendix
\ref{generalized-eq-diff-forms}, the usual Cartan map, as defined in
(\ref{cartan-map}), extends naturally to the equivariant
differential forms with coefficients in the ring of formal power
series. Moreover, it follows easily from Theorem
\ref{extended-Cartan-theorem} that $e^{-2i\mu} \sigma(\rho)\wedge
\overline{\rho}$ gets mapped to an ordinary closed differential form
\[e^{-2i \mu^lc_l}\sigma(\tau )\wedge \overline{\tau},\] where
$\tau$ is the unique differential form on $M/T$ such that $\pi^*\tau
=\rho_{\text{hor}}$, and $\pi: M\rightarrow M/T$ is the quotient
map.

Let $j_a : M_a \rightarrow M/T$ be the inclusion map and let $\pi_a:
\mu^{-1}(a) \rightarrow M_a$ be the natural projection map. From the
commutative diagram
\[ \begin{CD}
\mu^{-1}(a)  @>\text{ inclusion}>>
M \\
@V\pi_aVV             @V\pi VV\\
M_a @>j_a>> M/T
\end{CD}\]

we get that $j_a^*\tau =\widetilde{\rho}$. The closed three form $H$ descends to a closed three form $\widetilde{H}$ on $M/T$
such that $\pi^*\widetilde{H}=H$ and such that the quotient generalized Calabi-Yau structure $\rho_a$ is
$d_{\widetilde{H}}$-closed.  By Lemma \ref{anti-automorphism2} $\sigma(\rho_a)$ is
$d_{-\widetilde{H}}$-closed. Note that $d_{\widetilde{H}}$ is a real operator and
$\overline{\widetilde{\rho}}$ is $d_{\widetilde{H}}$-closed. Hence $\sigma({\rho}_a)\wedge \rho$ is
an ordinary closed differential form, representing a cohomology class in $H(M/T,\C)$.
Now let $\varphi$ be the
cohomology class of $[e^{-2i \mu^lc_l}\sigma(\tau )\wedge
\overline{\tau}]$ in $H(M/T,\C)$. Then we have

\[\begin{split} [\sigma(\widetilde{\rho}_a)\wedge
\overline{\widetilde{\rho}}_a]&=[j_a^*\left(\sigma(\tau)\wedge
\overline{\tau}\right)]\\&=[j_a^*\left(e^{2i\mu^lc_l}\wedge(e^{-2i\mu^lc_l}\wedge\sigma(\tau)\wedge\overline{\tau})\right)]\\
&=[e^{2ia^l\widetilde{c}_l}\wedge j_a^*(e^{-2i\mu^lc_l}\wedge\sigma(\tau)\wedge\overline{\tau})]\\&= [e^{2i a^l \widetilde{c}_l}]\wedge j_a^* \varphi,
\end{split}\]
where $\rho_a$ is the quotient generalized Calabi-Yau structure on
$M_a$, and $\widetilde{c}_l=j_a^*c_l$. Since $M_a$ is compact and
oriented, the embedding $j_a:M_a \rightarrow M/T$ defines an
integral homology class $[M_a] \in H_{2n-2k}(M/T,\Z)$. This homology
class depends smoothly on $a$ and is an integral class. So it is
actually independent of $a$. Let us fix an $a_0$ in the moment map
image $\mu(M)$. Then we have $[M_{a_0}]=[M_a]$. Thus up to a
normalization factor the integral (\ref{DHfunction}) can be
interpreted topologically as the pairing of the constant homology
class $[M_{a_0}]$ with the cohomology class
\begin{equation} \label{DH-function-expression} [e^{2i a^l \widetilde{c}_l}]\wedge j_a^*
\varphi.\end{equation}

Note that $\widetilde{c}_l$ represents a degree two cohomology
class, $1 \leq l \leq k$. Therefore the density function
(\ref{DHfunction}) is a polynomial function of degree at most $n-k$.
\end{proof}

\begin{remark} If the generalized Calabi-Yau structure $\rho$ is of
constant type $p$, then by the generalized Darboux theorem locally
it can always be written as $e^{B+i\omega}\theta_1\wedge\cdots
\wedge \theta_p$, where $B+i\omega$ is a closed complex valued two
form, and $\theta_1,\cdots, \theta_p$ are complex valued one forms.
(We refer to \cite[Thm. 4.35]{Gua07} for details.)  Combining this
with (\ref{DH-function-expression}), we conclude that the
Duistermaat-Heckman function is of degree at most $n-k-p$ provided
the generalized Calabi-Yau structure $\rho$ is of constant type $p$.

\end{remark}

\section{ Examples of Hamiltonian actions on generalized
Calabi-Yau manifolds} \label{eg-hamiltonian-g-calabi-yau}


The following theorem extends a useful construction of Hamiltonian
actions in symplectic geometry \cite[Prop. 4.2]{L04} to generalized
complex geometry, and it is an equivariant version of a construction
proposed in \cite[Thm. 2.2]{Ca05}. It allows us to construct
examples of Hamiltonian actions on compact twisted generalized
Calabi-Yau manifolds which are not a direct product of a symplectic
manifold and a complex manifold.
\begin{theorem}\label{Hamiltonian-G-Calabi-Yau} Let $(N,\J)$ be a compact $H$-twisted generalized
Calabi-Yau manifold. Then there exists a $S^2$ bundle $\pi:
M\rightarrow N$ which satisfies \begin{itemize} \item [a)] $M$
admits a $\pi^*H$-twisted generalized Calabi-Yau structure $\J'$;
\item[b)] there exists a Hamiltonian $S^1$ action on the generalized Calabi-Yau manifold
$(M,\J')$.
\end{itemize}
\end{theorem}

\begin{proof} Let $S^2$ be the set of unit vectors in $\R^3$. In
cylindrical coordinates $(\theta,h)$ away from the poles, $0 \leq
\theta < 2\pi, -1 \leq h \leq 1$, the standard symplectic form on
$S^2$ is the area form $\sigma=d\theta\wedge dh$. The circle $S^1$
acts on $S^2$ by rotations \[ e^{it}(\theta,h)=(\theta+t,h).\] This
action is Hamiltonian with the moment map given by $\mu=h$, i.e.,
the height function.

Let $\pi_P:P\rightarrow N$ be the principal $S^1$-bundle with Euler
class $[c]\in H^2(N, \Z)$, let $\Theta$ be the connection $1$-form
such that $\pi^*_Pc=d\Theta$, and let $M$ be the associated bundle
$P\times_{S^1}S^2$. Then $\pi: M\rightarrow N$ is a symplectic
fibration over the compact $H$-twisted generalized complex manifold
$N$. The standard symplectic form $\sigma$ on $S^2$ gives rise to a
symplectic form $\sigma_x$ on each fibre $\pi^{-1}(x)\cong S^2$,
where $x \in N$, whereas the height function $h$ on $S^2$ gives rise to a
global function $F$ on $M$ whose restriction to each fiber is just $h$.

Next we resort to minimal coupling construction to get a closed two
form $\eta$ which restricts to $\sigma_x$ on each fiber
$\pi^{-1}(x)$. Let us give a sketch of this construction here and
refer to \cite{GS84} for technical details. Consider the closed tow
form $-d(t\Theta)=-td\Theta-dt\wedge \Theta$ defined on $P\times
\R$. It is easy to check that the $S^1$ action defined by
\[ e^{i\theta}(p,t)=(e^{i\theta} p,t)\] is Hamiltonian with moment map
$t$. Thus the diagonal action of $S^1$ on $(P\times\R)\times S^2$ is
also Hamiltonian, and $M$ is just the reduced space of
$(P\times\R)\times S^2$ at the zero level. As a result, the closed
two form $(-d(t\Theta)+\sigma)\mid_{\text{zero level}}$ descends to
a closed two form $\eta$ on $M$ which restricts to $\sigma_x$ on each
fibre $\pi^{-1}(x)\cong S^2$, and the function \[(P\times \R)\times S^2 \rightarrow
\R, (p,t,z) \mapsto h(z)\] descends to a  function $F$ on $M$ which
restricts to the height function $h$ on each fibre $S^2$. Moreover,
there is another $S^1$ action on $(P\times \R)\times S^2$ which is defined
by letting $S^1$ act on $(P\times \R)$ trivially and act on $S^2$ by rotation;
this action commutes with the above diagonal $S^1$ action and so descends to a
fibrewise $S^1$ action on $M$. Let $X$ be the fundamental vector
field generated by this fibrewise $S^1$ action. It is easy to check
that the function $F$ is invariant under the fibrewise $S^1$ action
and $\iota_X (\eta)= dF$.

Let $\rho$ be the closed nowhere vanishing pure spinor associated to
the generalized Calabi-Yau structure $\J$ on $N$, and let
$(\cdot,\cdot)_M$ be the Mukai pairing on $M$. We claim that for
sufficiently small $\epsilon>0$,
\[(e^{i\epsilon \eta}\wedge \pi^*\rho, e^{-i\epsilon
\eta}\wedge \pi^*\overline{\rho})_M \neq 0.\]

Since $\eta$ is non-degenerate on the vertical tangent space
$\text{ker}\pi_*$, at each point $x \in M $ it determines a
horizontal subspace
\[\text{Hor}_x=\{ X\in T_xM: \eta(X,Y)=0, \,\forall Y\in
T_xS^2\}.\] The subspace $\text{Hor}_x$ is a complement to $T_xS^2$
and is isomorphic to $T_{\pi(x)}N$ via $\pi_*$. So we have a
spitting
\[ T_xM=\text{Hor}_x\oplus T_x S^2 \] which further induces a
splitting of the space of differential forms of degree two
\[ \wedge^2T_x^*M=(\wedge^2\text{Hor}_x^*) \oplus (\wedge^2T_x^*S^2)\oplus
(T_x^*M\otimes \text{Hor}^*_x).\] Here $\text{Hor}^*_x$ denotes the
dual space of $\text{Hor}_x$. By the definition of the horizontal
subspaces, $\eta$ splits into a direct sum $\eta=\eta_1+\eta_2$ such
that $\eta_{1,x} \in \wedge^2\text{Hor}_x^*, \eta_{2,x}\in \wedge^2
T_x^*S^2$. Let $(\cdot,\cdot)_{\text{Hor}}$ be the Mukai pairing on
$\text{Hor}$ and $(\cdot,\cdot)_N$ the Mukai pairing on $N$. Since
$\text{Hor}_x$ is isomorphic to $T_{\pi(x)}N$ via $\pi_*$,
$(\pi^*\rho,\pi^*\overline{\rho})_{\text{Hor}}=
(\rho,\overline{\rho})_N\neq 0$. In view of the compactness of $N$,
for sufficiently small $\epsilon$ we have
\[ (e^{i\epsilon \eta}\wedge \pi^*\rho, e^{-i\epsilon
\eta}\wedge \overline{\pi^*\rho})_M=(2\epsilon
i\eta_2)\wedge(e^{i\epsilon \eta_1}\pi^*\rho, e^{-i\epsilon
\eta_1}\pi^*\overline{\rho})_{\text{Hor}}\neq 0 .
 \]   It is straightforward to check that $e^{i\epsilon \eta}\wedge
 \pi^*\rho$ is a $d_{\pi^*H}$-closed form. Therefore
$e^{i\epsilon \eta}\wedge
 \pi^*\rho$ determines a $\pi^*H$-twisted generalized Calabi-Yau structure
 $\J'$.
Finally we note that
\[ (X-i\epsilon dF)\cdot ( e^{i\epsilon \eta}\wedge \pi^*\rho)=( i\iota_X (\epsilon \eta)-i\epsilon dF)
 \wedge( e^{i\epsilon \eta}\wedge \pi^*\rho)=0.
\]

This proves that the $X-i\epsilon dF \in C^{\infty}(L)$, where $L$
is the $i$-eigenbundle of $\J'$. Thus the $S^1$ action is
Hamiltonian with the generalized moment map $\epsilon F$ and trivial
moment one form.

\end{proof}

\begin{remark}\label{quotient-g-calabi-yau} \begin{itemize} \item [a)] Topologically, $M=P\times_{S^1}S^2$ depends only on the choice of
integral cohomology class $[c]\in H^2(N)$. When $[c]\neq 0$, it is easy to see that
 the $S^2$-bundle $M$ is non-trivial, i.e., not a Cartesian product of $N$ an $S^2$.

\item[b)] It is useful to have the following explicit description of $\eta$.
Observe that $d\theta-\Theta$ is a basic form on $(P \times \R)
 \times S^2$. Its restriction to the zero level of $(P \times \R)
 \times S^2$ descends to a one form $\gamma$ on $M$ whose
restriction to each fibre $S^2$ is just $d\theta$. It is easy to see
that on the associated bundle $P\times_{S^1}(S^2-\text{\{two
poles\}})$ we actually have $\eta = F\pi^*_M c+ \gamma \wedge
 dF$. Therefore, the generalized complex quotient taken at the level
 $\epsilon <t<\epsilon$ is $M$ with the quotient generalized Calabi-Yau structure
 $e^{itc}\wedge \rho$.

 \end{itemize}

\end{remark}

\begin{example} Let $N=T^4\cong \C^2/\Z^4$ be the four dimensional
torus with periodic coordinates  $z_i=x_i+y_i$, $1\leq i \leq 2$,
and let \[c=dx_1\wedge dy_1, \,\, \rho_1=e^{-ic}dz_2,
\,\,\rho_2=dz_1\wedge dz_2.\] Then it is easy to check that both
$\rho_1$ and $\rho_2$ are generalized Calabi-Yau structures on $N$
which are of type $1$ and $2$ respectively.

 Let $P$ be the principal $S^1$-bundle
$P$ with Euler class $[c]$, $M=P\times_{S^1}S^2$ the $S^2$-bundle associated to $P$, and $\rho'_i:=e^{-i c}\pi^*\rho_i$
the generalized Calabi-Yau structures on $M$ as constructed in
Theorem \ref{Hamiltonian-G-Calabi-Yau}. By Remark \ref{quotient-g-calabi-yau} the
generalized complex quotient taken at the level set $- \epsilon <
t<\epsilon$ is $M$ with the quotient generalized Calabi-Yau
structure $\widetilde{\rho}'_i=e^{-itc}\wedge\rho_i$, $i=1,2$. Note that
$N$ is endowed with the orientation $(\widetilde{\rho}'_i, \overline{\widetilde{\rho}'}_i)$, where
$(\cdot,\cdot)$ stands for the Mukai pairing on $N$. For convenience, we will denote by $N^+$ the manifold $N$
with the orientation $dx_1\wedge dx_2\wedge dx_3\wedge dx_4$, and $N^-$ the manifold $N$ with the orientation
 $-dx_1\wedge dx_2\wedge dx_3\wedge dx_4$. As straightforward calculation shows that
the Duistermaat-Heckman function for $(M,\rho'_1)$ is
\[\begin{split} f_1(t)&=-\dfrac{\pi}{2}\int_{N^+} (e^{-i(t+1)c}dz_2,
e^{i(t+1)c}d\overline{z}_2)\\&=-\dfrac{\pi}{2}\int_{N^+} 4(t+1)dx_1\wedge
dy_1 \wedge dx_2\wedge dy_2 \\&=-2\pi(t+1) ;\end{split}  \] whereas
the Dustermaat-Heckman function for $(M,\rho'_2)$ is
\[\begin{split} f_2(t)
&=-\dfrac{\pi}{2}\int_{N^-} (e^{-itc}\wedge dz_1 \wedge
dz_2,e^{itc}\wedge d\overline{z}_1\wedge d\overline{z}_2)
\\&=-\dfrac{\pi}{2}\int_{N^-} -4 dx_1\wedge dy_1 \wedge dx_2\wedge dy_2 \\&= -2\pi .\end{split} \]

\end{example}

\appendix
\section{The equivariant cohomology with coefficients in the ring of
formal power series}\label{generalized-eq-diff-forms} In this
appendix, we introduce the equivariant differential forms with
coefficients in the ring of formal power series. Let an $n$
dimensional compact connected Lie group $G$ act on an $m$
dimensional manifold $M$, let $\frak{g}$ be the Lie algebra of $G$.
Choose a basis $\xi_1,\xi_2,\cdots,\xi_n$ of $\frak{g}$ and let
$x_1, x_2,\cdots, x_n$ denote the corresponding dual basis in
$\frak{g}^*$ which generate the polynomial ring $S\frak{g}^*$. Let
$R=\C[[x_1,x_2,\cdots,x_n]]$ be the ring of formal power series over
$\C$.

\begin{definition} Define
\[ \widehat{\Omega}_G(M)=\left( \C[[x_1,x_2,\cdots,x_n]]\otimes \Omega(M) \right)^G, \]
to be the space of  equivariant differential forms with coefficients
in $R$, and define the equivariant differential
\begin{equation}\label{generalized-equivariant-differential} d_G (\sum_{I}x^I\otimes \alpha_I) =\sum_I
\left(x^I\otimes d\alpha_I+ (x^jx^I)\otimes
\iota_{\xi_{j,M}}\alpha_I\right),\,\,\,\sum_{I}x^I\otimes
\alpha_I\in (R\otimes\Omega(M))^G,\end{equation} where $I$ is a
multi-index, and $\xi_{j,M}$ denotes the vector field induced by
$\xi_j \in \frak{g}$. It is easy to check that $d_G^2=0$. Define the
cohomology
\[H(\widehat{\Omega}_G(M), d_G)=\text{ker} d_G/ \text{im} d_G\]
to be the equivariant cohomology with coefficients in $R$.
\end{definition}

 Let us give a more intrinsic description of $\widehat{\Omega}_G(M)$, the space of
 equivariant differential forms with coefficients in $\C[[x_1,x_2,\cdots,x_n]]$.
 Let $a=(x_1,x_2,\cdots,x_n)$ be the ideal in
 $S\frak{g}^*=\C[x_1,x_2,\cdots,x_n]$ generated by
 $x_1,x_2,\cdots,x_n$. Then the sequence of ideals
 \[S\frak{g}^*=a^0 \supset a \supset \cdots \supset a^k \supset \cdots\]
defines an $a$-adic topology on the polynomial ring $S\frak{g}^*$,
c.f. \cite[Chapter 10]{AM99}, and
$R=\C[[x_1,x_2,\cdots,x_n]]=\widehat{S\frak{g}^*}$ is the completion
of $S\frak{g}^*$ under the $a$-adic topology. Next consider the
$S\frak{g}^*$-module $S\frak{g}^*\otimes \Omega(M)$. The sequence of
modules
\[ S\frak{g}^*\otimes \Omega(M) \supset a \left(S\frak{g}^*\otimes \Omega(M)\right)
\supset \cdots \supset a^k \left(S\frak{g}^*\otimes
\Omega(M)\right)\supset \cdots \] defines an $a$-adic topology on
$S\frak{g}^*\otimes \Omega(M)$. So $(S\frak{g}^*\otimes
\Omega(M))^G$ inherits a topology from $S\frak{g}^*\otimes\Omega(M)$
as a subspace. Henceforth we will topologize $(S\frak{g}^*\otimes
\Omega(M))^G$  as a topological subspace of $S\frak{g}^*\otimes
\Omega(M)$ with the $a$-adic topology.  Note that
$\C[[x_1,x_2,\cdots,x_n]] \otimes \Omega(M)$ is the completion of
$S\frak{g}^*\otimes \Omega(M)$ under the $a$-adic topology. Also
note that an element $\sum_Ix^I\otimes \alpha_I\in
(\C[[x_1,x_2,\cdots,x_n]] \otimes \Omega(M))^G $ if and only if each
homogenous component $x^I\otimes \alpha_I \in (S\frak{g}^*\otimes
\Omega(M))^G$. Thus we conclude that $\widehat{\Omega}_G(M)$ is the
completion of $\Omega_G(M)=(S\frak{g}^*\otimes \Omega(M))^G$. After
these preparatory remarks, we are ready to state the following
lemma.

\begin{lemma}\label{continuous maps} Assume that $G$ acts freely on the manifold $M$.  Then
 the usual equivariant differential $d_G:\Omega_G\rightarrow
\Omega_G$, the chain homotopy operator $Q:\Omega_G\rightarrow
\Omega_G$ as defined in (\ref{chain-homotopy}), and the Cartan map
$\mathcal{C}:\Omega_G\rightarrow \Omega_{\text{bas}}$ as defined in
(\ref{cartan-map}) are continuous map from $\Omega_G$ to itself.
\end{lemma}

\begin{proof} We divide our proofs into three steps.
 \begin{itemize} \item[1)]  $d_G:\Omega_G\rightarrow
\Omega_G$ is continuous.

Note that as an operator $d_G$ is also defined on
$S\frak{g}^*\otimes \Omega(M)$. Furthermore, we have \[d_G
\left(a^k( S\frak{g}^*\otimes \Omega(M))\right) \subset a^k(
S\frak{g}^*\otimes \Omega(M)), \] for any $k\geq 0$. So
$d_G:S\frak{g}^*\otimes \Omega(M) \rightarrow S\frak{g}^*\otimes
\Omega(M)$ is continuous and so its restriction to $\Omega_G$ is
also continuous.

\item [2)] The chain homotopy map $Q$ is continuous.

Let $K$, $R$ and $F$ be the operators we introduced in Section
\ref{Equivariant de Rham theory}. Note that $K$, $RF$ and $Q$ are
all defined on $S\frak{g}^*\otimes \Omega(M)$ as well, and we have
\[ K\left(a^k( S\frak{g}^*\otimes \Omega(M))\right)\subset a^{k-1}( S\frak{g}^*\otimes
\Omega(M)),\]
\[ (RF)^i\left(a^k( S\frak{g}^*\otimes \Omega(M))\right) \subset a^{k-i}( S\frak{g}^*\otimes
\Omega(M)),\] where $0\leq i\leq k$. Thus $K$ and $RF$ are both
continuous operators and so their restrictions to $\Omega_G$ are
also continuous. Note that given any $\gamma \in S\frak{g}^*\otimes
\Omega(M)$, $R=d\theta^r\partial_r$ increases the form degree of
$\gamma$ by two, whereas $F$ preserves both the form degree and
polynomial degree of $\gamma$. So for $i>\frac{1}{2}\text{dim}M$ we
have $(RF)^i=0$ for dimension reasons. Therefore we conclude that
$Q=KF\left(RF+(RF)^2+\cdots\right)$ is also continuous on
$S\frak{g}^*\otimes\Omega(M)$ and so the restriction of $Q$ is
continuous on $\Omega_G$.

\item [3)] The Cartan map $\mathcal{C}:\Omega_G\rightarrow
\Omega_{\text{bas}}$ is continuous.

This follows from the simple observation that, for dimension reasons, by (\ref{cartan-map})
we have $\mathcal{C}\left(a^k( S\frak{g}^*\otimes
\Omega(M)\right)=0$ if $k> \frac{1}{2}\text{dim}M$.

\end{itemize}

\end{proof}

As we have explained, $\widehat{\Omega}_G(M)$ is the completion of
$\Omega_G(M)$. Thus $d_G$, $Q$, and $\mathcal{C}$ have an unique
extension to $\widehat{\Omega}_G(M)$. One checks easily that the
extension of $d_G$ to $\widehat{\Omega}_G$ is the one as we defined
in Definition \ref{generalized-equivariant-differential}, and the
extension of the Cartan map $\mathcal{C}:\widehat{\Omega}_G
\rightarrow \Omega_{\text{bas}}$ is given by
\begin{equation}\label{extended-Cartan-map} \sum_{I}x^I\otimes \alpha_I
\rightarrow \sum_I c^I\wedge \alpha_I,\end{equation} where $c^i$ is
the curvature elements as defined in (\ref{curvture-elements}), and
$I=(i_1,i_2,\cdots,i_n)$ is a multi-index. Note that if $\vert \,I
\,\vert =\sum_{k=1}^n i_k >\frac{1}{2}\text{dim }M$, then $c^I\wedge
\alpha_I=0$ for dimension reasons. So the right hand side of
(\ref{extended-Cartan-map}) is actually a finite sum. Since the
equality $d_GQ+Qd_G=I-\mathcal{C}$ holds on $\Omega_G$, it also
holds on the completion of $\Omega_G$, i.e., $\widehat{\Omega}_G$.

In summary, we have proved the following result.

\begin{theorem}\label{extended-Cartan-theorem}  Assume the action of $G$ on $M$ is free. Then the Cartan
map (\ref{extended-Cartan-map}) induces a natural isomorphism from
$H(\widehat{\Omega}_G(M), d_G)$ to the ordinary cohomology $H(M/G)$.
\end{theorem}

Indeed, we can prove a slightly more general result. Suppose that $H \in \Omega^3(M)$ is a closed basic three form.
Then the map \[H\wedge: \Omega_G \rightarrow \Omega_G,\,\,\, \alpha \mapsto H\wedge \alpha\] is obviously continuous. Since $H$ is a basic form, a straightforward check
shows that $Q (H\wedge \alpha) +H\wedge ( Q\alpha)=0$ for any $\alpha \in \Omega_G$. As a result, $d_{G,H}Q+Qd_{G,H}=I-\mathcal{C}$ holds on $\Omega_G$, and so it also holds on $\widehat{\Omega}_G$. This proves the following result.

\begin{theorem}\label{twisted-version-extended-Cartan-theorem}    Assume the action of $G$ on $M$ is free. Let
$\pi: M \rightarrow M/G$ be the quotient map. And assume that $H\in
\Omega^3(X)$ is a basic form so that there is a three form
$\widetilde{H} \in \Omega^3(X/G)$ which satisfies
$\pi^*\widetilde{H}=H$. Then the Cartan map (\ref{extended-Cartan-map}) induces a natural isomorphism from
$H_G(M, H)$ to the twisted cohomology $H(M/G, \widetilde{H})$.


\end{theorem}







\end{document}